\documentclass[reqno]{amsart}

\usepackage{amsmath}
\usepackage{amssymb}
\usepackage{amsfonts}
\usepackage{amsthm}
\usepackage[foot]{amsaddr}
\usepackage{mathtools}
\mathtoolsset{%
above-intertext-sep = -1ex
}

\usepackage[utf8]{inputenc}
\usepackage[T1]{fontenc}

\usepackage[sf,mono=false]{libertine}

\usepackage{dsfont}

\usepackage[%
cal=cm,
]
{mathalfa}

\usepackage[dvipsnames,svgnames]{xcolor}
\colorlet{MyBlue}{DodgerBlue!75!Black}
\colorlet{MyGreen}{DarkGreen!85!Black}

\usepackage[font=small,labelfont=bf]{caption}
\usepackage{subfigure}
\usepackage{tikz}
\usetikzlibrary{calc,decorations.pathmorphing,patterns}
\usepackage{tikz-cd}

\usepackage{acronym}
\usepackage{latexsym}
\usepackage{paralist}
\usepackage{wasysym}
\usepackage{xspace}

\usepackage[numbers,sort&compress]{natbib}

\usepackage{hyperref}
\hypersetup{
colorlinks=true,
linktocpage=true,
pdfstartview=FitH,
breaklinks=true,
pdfpagemode=UseNone,
pageanchor=true,
pdfpagemode=UseOutlines,
plainpages=false,
bookmarksnumbered,
bookmarksopen=false,
bookmarksopenlevel=1,
hypertexnames=true,
pdfhighlight=/O,
urlcolor=MyBlue!60!black,linkcolor=MyBlue!70!black,citecolor=DarkGreen!70!black, 
pdftitle={},
pdfauthor={},
pdfsubject={},
pdfkeywords={},
pdfcreator={pdfLaTeX},
pdfproducer={LaTeX with hyperref}
}

\newcommand{\EMAIL}[1]{\email{\href{mailto:#1}{#1}}}

\numberwithin{equation}{section}  
\usepackage[sort&compress,capitalize,nameinlink]{cleveref}
\newcommand{\crefrangeconjunction}{\textendash\hspace*{1pt}}

\newcommand{\dd}{\:d}

\newcommand{\eps}{\varepsilon}
\newcommand{\from}{\colon}

\newcommand{\pd}{\partial}
\newcommand{\wilde}{\widetilde}

\newcommand{\mg}{\succ}
\newcommand{\mgeq}{\succcurlyeq}

\newcommand{\mleq}{\preccurlyeq}

\newcommand{\bX}{\mathbf{X}}
\newcommand{\bY}{\mathbf{Y}}

\newcommand{\R}{\mathbb{R}}

\newcommand{\N}{\mathbb{N}}

\DeclareMathOperator*{\aff}{aff}
\DeclareMathOperator*{\argmax}{arg\,max}
\DeclareMathOperator*{\argmin}{arg\,min}
\DeclareMathOperator*{\intersect}{\bigcap}

\DeclareMathOperator*{\union}{\bigcup}
\DeclareMathOperator*{\Sym}{Sym}

\DeclareMathOperator{\bd}{bd}
\DeclareMathOperator{\bigoh}{\mathcal{O}}

\DeclareMathOperator{\dom}{dom}

\DeclareMathOperator{\ex}{\mathbb{E}}

\DeclareMathOperator{\hess}{Hess}

\DeclareMathOperator{\im}{im}

\DeclareMathOperator{\intr}{int}
\DeclareMathOperator{\one}{\mathds{1}}

\DeclareMathOperator{\prob}{\mathbb{P}}
\DeclareMathOperator{\rank}{rank}

\DeclareMathOperator{\relint}{ri}

\DeclareMathOperator{\tr}{tr}

\DeclarePairedDelimiter{\braces}{\{}{\}}
\DeclarePairedDelimiter{\bracks}{[}{]}
\DeclarePairedDelimiter{\parens}{(}{)}

\DeclarePairedDelimiter{\abs}{\lvert}{\rvert}
\DeclarePairedDelimiter{\norm}{\lVert}{\rVert}
\DeclarePairedDelimiterXPP{\dnorm}[1]{}{\lVert}{\rVert}{_{\ast}}{#1}
\DeclarePairedDelimiterXPP{\fnorm}[1]{}{\lVert}{\rVert}{_{F}}{#1}

\DeclarePairedDelimiterX{\braket}[2]{\langle}{\rangle}{#1\mathopen{}\hspace{1pt}\delimsize\vert\hspace{1pt}\mathopen{}#2}

\DeclarePairedDelimiterX{\inner}[2]{\langle}{\rangle}{#1,#2}
\DeclarePairedDelimiterX{\setdef}[2]{\{}{\}}{#1:#2}

\DeclarePairedDelimiterXPP{\probof}[1]{\prob}{(}{)}{}{%

#1}

\DeclarePairedDelimiterXPP{\exof}[1]{\ex}{[}{]}{}{%

#1}

\DeclarePairedDelimiterXPP{\trof}[1]{\tr}{[}{]}{}{#1}

\newcommand{\txs}{\textstyle}
\newcommand{\textpar}[1]{\textup(#1\textup)}

\newcommand{\insum}{\sum\nolimits}

\newcommand{\as}{\textup(a.s.\textup)\xspace}

\usepackage[textwidth=30mm]{todonotes}
\usepackage{soul}
\setstcolor{red}
\sethlcolor{SkyBlue}

\usepackage[norelsize,algoruled,vlined]{algorithm2e}

\SetKwBlock{Repeat}{Repeat}{}

\theoremstyle{plain}
\newtheorem{theorem}{Theorem}
\newtheorem{corollary}[theorem]{Corollary}
\newtheorem*{corollary*}{Corollary}
\newtheorem{lemma}[theorem]{Lemma}
\newtheorem{proposition}[theorem]{Proposition}

\theoremstyle{definition}
\newtheorem{definition}[theorem]{Definition}
\newtheorem*{definition*}{Definition}

\newtheorem{example}{Example}

\theoremstyle{remark}
\newtheorem{remark}{Remark}
\newtheorem*{remark*}{Remark}
\newtheorem*{notation*}{Notational remark}

\numberwithin{theorem}{section}
\numberwithin{remark}{section}
\numberwithin{example}{section}

\newenvironment{proofof}[1]{\begin{proof}[#1]}{\end{proof}}

\newcommand{\vecspace}{\mathcal{V}}
\newcommand{\dspace}{\vecspace^{\ast}}

\newcommand{\dvec}{u}

\newcommand{\feas}{\mathcal{X}}
\newcommand{\intfeas}{\feas^{\circ}}

\newcommand{\dual}{\mathcal{Y}}

\newcommand{\base}{p}

\newcommand{\obj}{f}
\newcommand{\sol}{x^{\ast}}

\newcommand{\dsol}{y^{\ast}}

\newcommand{\subd}{\partial}

\newcommand{\cone}{P}

\newcommand{\tcone}{\mathrm{TC}}

\newcommand{\pcone}{\mathrm{PC}}

\newcommand{\breg}{D}
\newcommand{\mirror}{Q}
\newcommand{\depth}{\Omega}
\newcommand{\fench}{F}

\DeclareMathOperator{\Eucl}{\Pi}
\DeclareMathOperator{\logit}{\Lambda}
\DeclareMathOperator{\mlogit}{\boldsymbol{\Lambda}}

\newcommand{\act}{x}
\newcommand{\actions}{\mathcal{X}}
\newcommand{\acts}{\actions}

\newcommand{\payv}{v}

\newcommand{\cost}{c}
\newcommand{\Cost}{C}

\newcommand{\graph}{\mathcal{G}}
\newcommand{\vertices}{\mathcal{V}}
\newcommand{\edges}{\mathcal{E}}

\newcommand{\edge}{e}
\newcommand{\edgealt}{\edge'}

\newcommand{\source}{o}

\newcommand{\sink}{d}

\newcommand{\rate}{\lambda}

\newcommand{\flow}{x}
\newcommand{\flows}{\feas}

\newcommand{\load}{w}

\newcommand{\nRoutes}{P}
\newcommand{\routes}{\mathcal{\nRoutes}}
\newcommand{\route}{p}
\newcommand{\routealt}{\route'}

\newcommand{\argdot}{\,\cdot\,}

\newcommand{\filter}{\mathcal{F}}
\newcommand{\gen}{\mathcal{L}}

\newcommand{\simplex}{\Delta}
\newcommand{\intsimplex}{\simplex^{\!\circ}}
\newcommand{\spectron}{\boldsymbol{\mathcal{D}}}
\newcommand{\intspectron}{\spectron^{\circ}}

\newcommand{\step}{\gamma}
\newcommand{\temp}{\eta}

\newcommand{\lyap}{V}

\newcommand{\noise}{Z}
\newcommand{\snoise}{\xi}
\newcommand{\mean}{\mu}
\newcommand{\noisevol}{\sigma}
\newcommand{\noisedev}{\noisevol_{\ast}}
\newcommand{\noisevar}{\noisedev^{2}}
\newcommand{\covmax}{\covmat_{\max}}
\newcommand{\covmat}{\Sigma}
\newcommand{\semiflow}{\Phi}
\newcommand{\dfeas}{\norm{\feas}}
\newcommand{\olimit}{\hat x}
\newcommand{\genvec}{u}
\newcommand{\genvecs}{\mathcal{U}}
\newcommand{\ybase}{y}

\newcommand{\strong}{\alpha}
\newcommand{\test}{\phi}
\newcommand{\Lip}{L}
\newcommand{\fradius}{R}

\newcommand{\ball}{\mathbb{B}}

\newcommand{\cpt}{C}
\newcommand{\nhd}{L}

\newcommand{\occ}{\mu}

\newcommand{\subspace}{\vecspace_{0}}
\newcommand{\dsubspace}{\subspace^{\ast}}
\newcommand{\subdual}{\dual_{0}}
\newcommand{\subcpt}{\cpt_{0}}
\newcommand{\submirror}{\mirror_{0}}

\newcommand{\dincl}{\pi_{0}}

\begin{document}


\title
[Gradient-like flows with noisy input]
{On the convergence of gradient-like\\flows with noisy gradient input}

\author{Panayotis Mertikopoulos$^{\ast}$}
\address
{$^{\ast}$\,%
Univ. Grenoble Alpes, CNRS, Inria, LIG, F-38000, Grenoble, France.}
\EMAIL{panayotis.mertikopoulos@imag.fr}

\author{Mathias Staudigl$^{\sharp}$}
\address{$^{\sharp}$\,%
Maastricht University, Department of Quantitative Economics, P.O. Box 616, NL\textendash 6200 MD Maastricht, The Netherlands.}
\EMAIL{m.staudigl@maastrichtuniversity.nl}

\thanks{
%
%
The authors are indebted to the associate editor and the two anonymous referees for their detailed suggestions and remarks.
PM was partially supported by the French National Research Agency (ANR) project ORACLESS (ANR\textendash GAGA\textendash13\textendash JS01\textendash 0004\textendash 01) and the Huawei Innovation Research Program ULTRON%
}


\newcommand{\acli}[1]{\textit{\acl{#1}}}
\newcommand{\acdef}[1]{\textit{\acl{#1}} \textup{(\acs{#1})}\acused{#1}}
\newcommand{\acdefp}[1]{\emph{\aclp{#1}} \textup(\acsp{#1}\textup)\acused{#1}}

\newacro{OD}[O/D]{origin-destination}
\newacro{uhc}{upper hemicontinuous}
\newacro{HR}{Hessian Riemannian}
\newacro{OU}{Orn\-stein\textendash Uhl\-en\-beck}
\newacro{FP}{Fokker\textendash Planck}
\newacro{GD}{gradient descent}
\newacro{LHS}{left-hand side}
\newacro{RHS}{right-hand side}
\newacro{SMD}{stochastic mirror descent}
\newacro{PMD}{perturbed mirror descent}
\newacro{NMD}{noisy mirror descent}
\newacro{CMD}{contaminated mirror descent}
\newacro{SDE}{stochastic differential equation}
\newacro{MDS}{martingale difference sequence}
\newacro{EW}{exponential weights}
\newacro{APT}{asymptotic pseudotrajectory}
\newacroplural{APT}{asymptotic pseudotrajectories}
\newacro{WAPT}{weak asymptotic pseudotrajectory}
\newacroplural{WAPT}{weak asymptotic pseudotrajectories}
\newacro{KKT}{Karush\textendash Kuhn\textendash Tucker}
\newacro{MSE}{mean squared error}
\newacro{DSC}{diagonal strict concavity}
\newacro{NE}{Nash equilibrium}
\newacroplural{NE}[NE]{Nash equilibria}
\newacro{ESS}{evolutionarily stable state}
\newacro{MD}{mirror descent}
\newacro{MA}{mirror ascent}
\newacro{SA}{stochastic approximation}
\newacro{VS}{variational stability}
\newacro{iid}[i.i.d.]{independent and identically distributed}

\begin{abstract}
%
%
In view of solving convex optimization problems with noisy gradient input, we analyze the asymptotic behavior of gradient-like flows under stochastic disturbances.
Specifically, we focus on the widely studied class of \acl{MD} schemes for convex programs with compact feasible regions, and we examine the dynamics' convergence and concentration properties in the presence of noise.
In the vanishing noise limit, we show that the dynamics converge to the solution set of the underlying problem \as.
Otherwise, when the noise is persistent, we show that the dynamics are concentrated around interior solutions in the long run, and they converge to boundary solutions that are sufficiently ``sharp''.
Finally, we show that a suitably rectified variant of the method converges irrespective of the magnitude of the noise (or the structure of the underlying convex program), and we derive an explicit estimate for its rate of convergence.
\end{abstract}

\maketitle

\renewcommand{\sharp}{\gamma}
\acresetall
\allowdisplaybreaks

\section{Introduction}
\label{sec:introduction}

Consider an unconstrained convex program of the form
\begin{equation}
\label{eq:program-unconstrained}
\tag{P$_{0}$}
\begin{aligned}
\textrm{minimize}
	&\quad
	\obj(x),
\end{aligned}
\end{equation}
where $\obj\from\vecspace\to\R$ is a convex function defined on some finite-dimensional real space $\vecspace$.
To solve \eqref{eq:program-unconstrained}, a key role is played by the \emph{gradient flow} of $\obj$, i.e. the \acl{GD} dynamics
\begin{equation}
\label{eq:GD}
\tag{GD}
\dot x
	= -\nabla\obj(x).
\end{equation}%
As is well known, under mild regularity assumptions for $\obj$, the solution trajectories of \eqref{eq:GD} converge to the solution set of \eqref{eq:program-unconstrained} \textendash\ provided of course that said set is nonempty.
Thus, building on this ``quick-and-easy'' convergence result, \eqref{eq:GD} and its variants have become the starting point for a vast corpus of literature in convex optimization and control.

Notwithstanding, if the gradient input to \eqref{eq:GD} is contaminated by noise (e.g. due to faulty measurements and/or other exogenous factors), this convergence is destroyed, even in simple, one-dimensional problems.
To see this, take $\obj(x) = \theta (x - \mean)^{2}/2$ with parameters $\mean\in\R$ and $\theta>0$, and consider the perturbed dynamics
\begin{equation}
\label{eq:SGD-simple}
dX
	= - \theta (X-\mean) \dd t + \noisevol \dd W,
\end{equation} 
where $W(t)$ is a one-dimensional Wiener process (Brownian motion) with volatility $\noisevol>0$.
This system describes an \ac{OU} process with mean $\mean$ and reversion rate $\theta$, leading to the explicit solution formula
\begin{equation}
\label{eq:OU}
X(t)
	= X(0) e^{- \theta t}
	+ \mean (1 - e^{-\theta t})
	+ \noisevol \int_{0}^{t} e^{-\theta(t - s)} \dd W(s).
\end{equation}
Thanks to this expression, several conclusions can be drawn regarding \eqref{eq:SGD-simple}.
First, even though the drift of the dynamics \eqref{eq:SGD-simple} vanishes at $\mean$ (and only at $\mean$), $X(t)$ \emph{does not} converge to $\mean$ with positive probability;
instead, $X(t)$ converges in distribution to a Gaussian random variable $X_{\infty}$ with mean $\mean$ and variance $\noisevol^{2}/(2\theta)$ \cite[Chap.~5.6]{KS98}.
Thus, in the long run, $X(t)$ will fluctuate around $\mean$ with a spread that is roughly proportional to the noise volatility coefficient $\noisevol$.

More generally, by solving the associated Fokker\textendash Planck equation, it is well known that the perturbed \acl{GD} system
\begin{equation}
\label{eq:SGD}
dX
	= -\nabla\obj(X) \dd t + \noisevol\dd W
\end{equation}
admits a unique invariant measure $e^{-2\obj(x)/\noisevol^{2}}\dd x$, which gives rise to a (unique) invariant distribution $d \occ_{\infty}\propto e^{-2\obj(x)/\noisevol^{2}}\dd x$ (assuming that $\int e^{-2\obj(x)/\noisevol^{2}} \dd x < \infty$ for normalization purposes).
In other words, for large $t$, $X(t)$ is most likely to be found near $\argmin\obj$ and this likelihood is (exponentially) inversely proportional to $\noisevol$.
Moreover by ergodicity, the distribution of the time-averaged process $\bar X(t) = t^{-1} \int_{0}^{t} X(s) \dd s$ also converges to $\occ_{\infty}$;
thus, in general, even the ergodic average of $X(t)$ fails to converge to $\argmin\obj$ with positive probability.

Somewhat surprisingly, except for these basic results for unconstrained problems, the long-run behavior of constrained gradient-like flows with noisy input remains largely unexplored.
With this in mind, we consider here the widely studied class of \acdef{MD} dynamics that were pioneered by Nemirovski and Yudin \cite{NY83} for constrained convex programs (and which include \acl{GD} as a special case), and we examine their convergence properties in the presence of stochastic disturbances.

Concretely, our paper focuses on constrained convex programs of the form
\begin{equation}
\label{eq:program}
\tag{P}
\begin{aligned}
\textrm{minimize}
	&\quad
	\obj(x),
	\\
\textrm{subject to}
	&\quad
	x\in\feas,
\end{aligned}
\end{equation}
where $\feas$ is a compact convex subset of $\vecspace$
and
$\obj\from\feas\to\R$ is a $C^{1}$-smooth convex function on $\feas$.
In continuous time, the dynamics of \acl{MD} take the form
\begin{equation}
\label{eq:MD-intro}
\tag{MD}
\begin{aligned}
\dot y
	&= -\nabla\obj(x),
	\\
x
	&= \mirror(\temp y),
\end{aligned}
\end{equation}
where, referring to \cref{sec:prelims} for the details,
$\temp>0$ is a sensitivity parameter
while
the ``mirror map'' $\mirror(y) = \argmax_{x\in\feas}\{\braket{y}{x} - h(x)\}$ is a projection-like mapping defined via a strongly convex ``prox-function'' $h\from\feas\to\R$.
In this way, \eqref{eq:MD-intro} is the continuous-time limit of Nesterov's well-known \emph{dual averaging} scheme \cite{Nes09}
\begin{equation}
\label{eq:MD-discrete}
\begin{aligned}
y_{t+1}
	&= y_{t} - \step_{t} \nabla\obj(x_{t}),
	\\
x_{t+1}
	&=\mirror(\temp y_{t+1}),
\end{aligned}
\end{equation}
where $\step_{t} > 0$, $t=1,2,\dotsc$, is a variable step-size sequence.

The dynamics of \acl{MD} have recently attracted considerable interest in optimization
\cite{ABB04,BT03,ABRT04,KBB15,WWJ16}
and machine learning \cite{KM17,Kri16},
and we summarize some of the convergence results obtained for \eqref{eq:MD} in \cref{sec:prelims}.
As an example, if $h(x) = \frac{1}{2}\norm{x}_{2}^{2}$, we have $\mirror(y) = \argmin_{x\in\feas} \norm{y-x}$, so \eqref{eq:MD-intro} boils down to a (Euclidean) projected \acl{GD} scheme.
Extending this interpretation to general $h$, the authors of \cite{ABB04,BT03,ABRT04} showed that \eqref{eq:MD-intro} may be viewed as the gradient flow of $\obj$ with respect to a certain Riemannian metric on $\feas$;
thus, in addition to projected (Euclidean) \acl{GD}, \eqref{eq:MD-intro} also covers a very broad class of Riemannian gradient-like flows (cf.~\cref{sec:discussion}).

Moving beyond this deterministic framework, the study of \acl{MD} with noisy first-order feedback is a classic topic in optimization (see e.g. \cite{Nes04,NJLS09,DAJJ12,LTE15} and references therein).
In view of this, our paper focuses on the \acli{SMD} dynamics
\begin{equation}
\label{eq:SMD-intro}
\tag{SMD}
\begin{aligned}
dY
	&= -\nabla\obj(X) \dd t + d\noise,
	\\
X
	&= \mirror(\temp Y),
\end{aligned}
\end{equation}
where $\noise(t)$ is an Itô martingale process (such as Brownian motion) representing the sum of all random disturbances affecting the gradient input to \eqref{eq:MD}.
In this stochastic setting, the simple example \eqref{eq:SGD-simple} shows that the deterministic convergence properties of \eqref{eq:MD-intro} cannot be carried over to \eqref{eq:SMD-intro} in full generality.
Accordingly, our paper focuses on the following questions:

\begin{enumerate}
\item
If the volatility of $\noise(t)$ decays over time,
intuition suggests that the good convergence properties of \eqref{eq:MD-intro} should also apply to \eqref{eq:SMD-intro}.
In \cref{sec:smallnoise}, we make this intuition precise by noting that the solutions of \eqref{eq:SMD-intro} correspond to \acdefp{APT} \cite{BH96} of \eqref{eq:MD-intro} in the vanishing noise limit.%
\footnote{For background information on the theory of stochastic approximation and \acp{APT}, see \cite{BH96,Ben99}.}
Except for the very recent paper \cite{BH16}, we are not aware of a similar \ac{APT}-based analysis in optimization, and this interesting link between deterministic and stochastic \acl{MD} only becomes transparent in continuous time.

\item
If the noise is persistent, trajectory convergence to interior points is no longer possible.
Nonetheless, if $\obj$ is \emph{strongly} convex and \eqref{eq:program} admits an interior solution $\sol$ (a case of particular interest in machine learning and statistics \cite{SNW12}), the long-run behavior of \eqref{eq:SMD-intro} can be described by examining the dynamics' invariant distribution.
Our analysis in \cref{sec:invariant} provides an explicit estimate for this invariant measure and shows that \eqref{eq:SMD-intro} spends an arbitrarily large fraction of the time arbitrarily close to $\sol$ if the dynamics' sensitivity parameter $\temp$ is small enough.

\item
Departing from the interior case, we also consider sharp solutions that arise e.g. in generic linear programs.
In this case, if the sensitivity parameter $\temp$ of \eqref{eq:SMD-intro} is taken sufficiently small, $X(t)$ converges \as and this convergence occurs in finite time if the mirror map $\mirror$ is surjective (cf. \cref{sec:sharp}).

\item
Finally, if no assumptions can be made on the structure of \eqref{eq:program}, we show in \cref{sec:transforms} that a suitably rectified variant of \eqref{eq:SMD-intro} with a decreasing sensitivity parameter converges with probability $1$.
Specifically, if $\temp\equiv\temp(t)$ decays as $\Theta(t^{-1/2})$, the ergodic average $\bar X(t) = t^{-1}\int_{0}^{t}X(s) \dd s$ of $X(t)$ enjoys  an almost sure $\bigoh(t^{-1/2}\sqrt{\log\log t})$ value convergence rate.%
\footnote{That is, $\obj(\bar X(t)) - \min\obj = \bigoh(t^{-1/2}\sqrt{\log\log t})$ except for a set of measure zero.}
\end{enumerate}

At a technical level, this paper belongs to the growing literature on dynamical systems that arise in the solution of continuous optimization problems and variational inequalities \textendash\ see e.g. \cite{HM96,NY83,ISdCN99,SBC14,AA15,BC16,SBC14,KBB15,WWJ16} and references therein.
More precisely, the deterministic bedrock of our analysis coincides with the gradient-like dynamics studied in \cite{BT03,ABRT04,ABB04};
along with an important dichotomy that arises in the stochastic regime, we make this link precise in \cref{sec:discussion}.
Otherwise, from a stochastic viewpoint, the work that is closest to our analysis is the recent paper \cite{RB13} where the authors showed that the ergodic average of an interior-valued subclass of \eqref{eq:SMD-intro} converges within $\bigoh(\noisevol^{2})$ of the solution set of \eqref{eq:program} and further provided a variance reduction scheme based on the parallel sampling of multiple trajectories.
To the best of our knowledge, this is the only result known for \eqref{eq:SMD-intro};
our analysis in \cref{sec:transforms} shows that this optimality gap can be reduced to $0$ if \eqref{eq:SMD-intro} is run with a decreasing sensitivity parameter.

There is also a broad and vigorous literature on second-order gradient systems such as Nesterov's accelerated gradient method (cf.~\cite{SBC14} and references therein) and Polyak's ``heavy ball with friction'' dynamics \cite{AGR00,CEG09,BBJ15,AABR02,Pol87}.
Up to a dissipative friction term, such systems can be seen as quasi-gradient flows on $\feas\times\vecspace$ (the system's phase space) and recent works have considered the limit behavior of Itô perturbations of such flows \cite{GP14}.
Even though they might share some asymptotic properties, these second-order systems are fundamentally different from the first-order systems that we consider here (even in the noiseless, deterministic regime), so there is no overlap of results.

\section{Preliminaries}
\label{sec:prelims}

\subsection*{Notation}

Given an $n$-dimensional real space $\vecspace$ with norm $\norm{\cdot}$,
we will write
$\dspace$ for its dual,
$\braket{y}{x}$ for the pairing between $y\in\dspace$ and $x\in\vecspace$,
and $\dnorm{y} \equiv \sup\setdef{\braket{y}{x}}{\norm{x} \leq 1}$ for the dual norm of $y$ in $\dspace$.
Also, given an extended-real-valued function $g\from\vecspace\to\R\cup\{+\infty\}$,
its \emph{effective domain} is defined as $\dom g = \setdef{x\in\vecspace}{g(x)<\infty}$
and 
its \emph{subdifferential} at $x\in\dom g$ is given by $\subd g(x) = \setdef{y\in\dspace}{\text{$g(x') \geq g(x) + \braket{y}{x' - x}$ for all $x'\in\vecspace$}}$.

In the rest of the paper, $\feas$ will denote a compact convex subset of $\vecspace$
and $\obj\from\feas\to\R$ will be a $C^{1}$-smooth convex function on $\feas$;
we will also write $\intfeas \equiv \relint(\feas)$ for the relative interior of $\feas$ and $\norm{\feas} = \max\setdef{\norm{x' - x}}{x,x'\in\feas}$ for its diameter.
For $x\in\feas$, the \emph{tangent cone} $\tcone_{\feas}(x)$ is the closure of the set of all rays emanating from $x$ and intersecting $\feas$ in at least one other point.
The \emph{polar cone} $\pcone_{\feas}(x)$ to $\feas$ at $x$ is then defined as $\pcone_{\feas}(x) = \setdef{y\in\dspace}{\braket{y}{z} \leq 0\; \text{for all}\; z\in\tcone_{\feas}(x)}$.
For concision, when $\feas$ is understood from the context, we will drop it altogether and we will write $\tcone(x)$ and $\pcone(x)$ instead.

Finally, the asymptotic equality notation ``$f(t) \sim g(t)$ for large $t$'' means that $\lim_{t\to\infty} f(t)/g(t) = 1$;
the symbols ``$\lesssim$'' and ``$\gtrsim$'' are defined analogously.

\subsection{Mirror descent}
\label{sec:MD}

Dating back to Nemirovski and Yudin \cite{NY83}, the main idea of \acl{MD} is as follows:
Given a smooth convex objective $\obj\from\feas\to\R$, the optimizer takes an infinitesimal step along the negative gradient of $\obj$ in the dual space $\dspace$;
the output is then ``mirrored'' back to the problem's feasible region $\feas\subseteq \vecspace$ and the process continues.
More precisely, in continuous time, the dynamics of this process can be represented as
\begin{equation}
\label{eq:MD}
\tag{MD}
\begin{aligned}
\dot y
	&= \payv(x),
	\\
x
	&= \mirror(\temp y),
\end{aligned}
\end{equation}
where:
\begin{enumerate}
[\indent 1.]
\addtolength{\itemsep}{.5ex}
\item
$\payv(x) = -\nabla\obj(x)$ denotes the negative gradient of $\obj$ at $x$.
\item
$y\in\dspace$ is an auxiliary ``score'' variable that aggregates gradient steps.
\item
$\temp>0$ is a sensitivity parameter (see below).
\item
$\mirror\from\dspace\to\feas$ is the \emph{mirror} (or \emph{choice}) map that outputs a solution candidate $x\in\feas$ as a function of the score variable $y\in\dspace$ (also discussed below).
\end{enumerate}

%
%

A key element in the above description of \acl{MD} is the distinction between primal and dual variables \textendash\ that is, between candidate solutions $x\in\feas$ and score variables $y\in\dspace$.
To emphasize this duality, we will write $\dual \equiv \dspace$ for the dual space of $\vecspace$ and, following \cite{Nes09}, we will often refer to the dynamics \eqref{eq:MD} as \emph{dual averaging}.
Also, in terms of regularity, we will assume that
\begin{equation}
\label{eq:Lipschitz}
\tag{$\mathbf{H}_{1}$}
\text{$\payv(x)$ is Lipschitz continuous on $\feas$.}
\end{equation}
Strictly speaking, Hypothesis \eqref{eq:Lipschitz} is not needed for much of the analysis of \eqref{eq:MD} and can be replaced e.g. by global integrability of $\payv$;
however, it simplifies the presentation considerably, so we keep it throughout our paper.

Given that the dual variable $y$ aggregates (negative) gradient steps, a reasonable candidate for the mirror map $\mirror$ might appear to be the $\argmax$ correspondence $y \mapsto \argmax_{x\in\feas} \braket{y}{x}$ whose output is most closely aligned with $y$.
However, this assignment is set-valued and generically selects only extreme points of $\feas$, so it is ill-suited for general, nonlinear convex programs.
On that account, \eqref{eq:MD} is typically run with ``regularized'' mirror maps of the form
$y \mapsto \argmax_{x\in\feas} \{ \braket{y}{x} - h(x) \}$
where the penalty term $h(x)$ satisfies the following:

\begin{definition}
\label{def:mirror}
We say that $h\from\feas\to\R$ is a \emph{regularizer} (or \emph{penalty function}) on $\feas$ if it is continuous and \emph{strongly convex}, i.e. there exists some $K>0$ such that
\begin{equation}
h(\lambda x + (1-\lambda) x')
	\leq \lambda h(x) + (1-\lambda) h(x')
	- \tfrac{1}{2} K \lambda (1-\lambda) \norm{x' - x}^{2},
\end{equation}
for all $x,x'\in\feas$ and all $\lambda\in[0,1]$.
The \emph{mirror map} induced by $h$ is then defined as
\begin{equation}
\label{eq:mirror}
\mirror(y)
	= \argmax_{ x\in\feas} \{\braket{y}{x} - h(x) \}.
\end{equation}
\end{definition}

In view of the above, we have $\mirror(\temp y) = \argmax_{x\in\feas}\{ \braket{y}{x} - \temp^{-1} h(x) \}$, so $\temp$ essentially controls the weight of the penalty term $h(x)$ in \eqref{eq:mirror}.
Consequently, as $\temp\to 0$, the ``$\temp$-deflated'' mirror map $\mirror(\temp y)$ tends to select points that are closer to the ``prox-center'' $x_{c} \equiv \argmin h$ of $\feas$ (implying in turn that the primal variable $x$ becomes less susceptible to changes in $y$, hence the name ``sensitivity'').

\smallskip

For concreteness, we discuss below some examples of this construction:

\begin{example}[Euclidean projections]
\label{ex:Eucl}
Let $h(x) = \frac{1}{2} \norm{x}_{2}^{2}$.
Then, $h$ is $1$-strongly convex with respect to $\norm{\cdot}_{2}$ and the induced mirror map is the closest point projection
\begin{equation}
\label{eq:mirror-Eucl}
\Eucl(y)
	= \argmax_{x\in\feas} \braces[\big]{ \braket{y}{x} - \tfrac{1}{2} \norm{x}_{2}^{2} }
	= \argmin_{x\in\feas}\; \norm{y - x}_{2}^{2}.
\end{equation}
The dynamics derived from \eqref{eq:mirror-Eucl} may thus be viewed as a continuous-time version of (Euclidean) projected gradient descent \cite{ABRT04,Nes09,Mer17}.
For future reference, we also note that $h$ is differentiable throughout $\feas$ and $\Eucl$ is \emph{surjective} (i.e. $\im\Eucl = \feas$).
\end{example}

\begin{example}[Entropic regularization]
\label{ex:logit}
Let $\simplex = \setdef{x\in\R^{n}_{+}}{\sum_{i=1}^{n} x_{i} = 1}$ denote the unit simplex of $\R^{n}$ and consider the (negative) Gibbs entropy
\begin{equation}
\label{eq:entropy}
h(x)
	= \sum_{i=1}^{n} x_{i} \log x_{i}.
\end{equation}
The function $h(x)$ is $1$-strongly convex with respect to the $L^{1}$-norm on $\R^{n}$ and a straightforward calculation shows that the induced mirror map is
\begin{equation}
\label{eq:mirror-logit}
\logit(y)
	= \frac{1}{\sum_{i=1}^{n} \exp(y_{i})} (\exp(y_{1}),\dotsc,\exp(y_{n})).
\end{equation}
This model is known as \emph{logit choice} and the associated dynamics have been studied extensively in linear programming \cite{Kar90}, online learning \cite{SS11} and game theory \cite{HS98}.
In contrast to \cref{ex:Eucl}, $h$ is differentiable \emph{only} on the relative interior $\intsimplex$ of $\simplex$ and $\im\logit = \intsimplex$ (i.e. $\logit$ is ``essentially'' surjective).
\end{example}

\begin{example}[Matrix regularization]
\label{ex:matrix}
Motivated by applications to semidefinite programming, consider the unit spectrahedron $\spectron = \setdef{\bX\in\Sym(\R^{n})}{\bX\mgeq 0,\, \tr\bX \leq 1}$ of positive-semidefinite matrices with nuclear norm $\norm{\bX}_{1} = \tr\bX \leq 1$.
A widely used regularizer on $\spectron$ is provided by the \emph{von Neumann entropy} \cite{Ved02}
\begin{equation}
\label{eq:entropy-vN}
h(\bX)
	= \tr(\bX \log \bX) + (1 - \tr\bX) \log(1 - \tr\bX),
\end{equation}
which is $(1/2)$-strongly convex with respect to the nuclear norm \cite{KSST12}.
A straightforward calculation \cite{MBNS17} then shows that the induced mirror map is given by
\begin{equation}
\label{eq:logit-matrix}
\mlogit(\bY)
	= \frac{\exp(\bY)}{1 + \norm{\exp(\bY)}_{1}}
	\quad
	\text{for all $\bY\in\mathrm{Sym}(\R^{n})$}.
\end{equation}
As in \cref{ex:Eucl}, $h$ is differentiable \emph{only} on the relative interior $\intspectron$ of $\spectron$;
furthermore, since $\exp(\bY) \mg 0$ for all $\bY\in\Sym(\R^{n})$, we have $\im\mlogit = \intspectron$ (i.e. $\mlogit$ is ``essentially'' surjective).
\end{example}

The examples above highlight an important relationship between the domain of differentiability of $h$ and the image of the induced mirror map $\mirror$.
To describe it in detail, extend $h$ to all of $\vecspace$ by setting $h\equiv\infty$ outside $\feas$,
and let $\dom\subd h \equiv \setdef{x\in\feas}{\subd h(x)\neq\varnothing}$ be the domain of subdifferentiability of $h$.
We then have the following characterization of $\mirror$:

\begin{proposition}
\label{prop:mirror}
Let $h$ be a $K$-strongly convex regularizer,
let $\mirror\from\dual\to\acts$ be the mirror map induced by $h$,
and let $h^{\ast}(y) = \max \setdef{\braket{y}{x} - h(x)}{x\in\feas}$ denote the convex conjugate of $h$.
Then:
\begin{enumerate}
[\indent\upshape 1)]
\item
$x=\mirror(y)$ if and only if $y\in\subd h(x)$;
in particular, $\im\mirror = \dom\subd h$.
\item
$h^{\ast}$ is differentiable on $\dual$ and $\nabla h^{\ast}(y) = \mirror(y)$ for all $y\in\dual$.
\item
$\mirror$ is $(1/K)$-Lipschitz continuous.
\end{enumerate}
\end{proposition}

\begin{proof}
Standard;
see e.g. \cite[Theorem 23.5]{Roc70} and \cite[Theorem 12.60(b)]{RW98}.
\end{proof}

Since $\intfeas \subseteq \dom\subd h \subseteq \feas$ \cite[Chap.~23]{Roc70}, \cref{prop:mirror} shows that $\mirror$ is ``almost'' surjective;
specifically, the only points of $\feas$ that do not belong to $\im\mirror$ are boundary points of $\feas$ where $h$ becomes ``infinitely steep''.
Motivated by this, we say that $h$ is \emph{steep} at $x$ if $\subd h(x) = \varnothing$ and \emph{nonsteep} otherwise.
As a result, regularizers that are everywhere nonsteep induce mirror maps that are surjective (\cref{ex:Eucl}),
while regularizers that are steep throughout $\bd(\feas)$ give rise to interior-valued mirror maps (\cref{ex:matrix}).

\subsection{Bregman divergences and the Fenchel coupling}
\label{sec:Fenchel}

Another key tool in the convergence analysis of \acl{MD} (at least when $h$ is steep) is the \emph{Bregman divergence} $\breg(\base,x)$ between $x\in\feas$ and a target point $\base\in\feas$.
Following \cite{Kiw97b}, $\breg(\base,x)$ is defined as the difference between $h(\base)$ and the best linear approximation of $h(\base)$ starting from $x$, viz.
\begin{equation}
\label{eq:Bregman}
\breg(\base,x)
	= h(\base) - h(x) - h'(x;\base-x),
\end{equation}
where $h'(x;z) = \lim_{t\to0^{+}} t^{-1} [h(x+tz) - h(x)]$ denotes the one-sided derivative of $h$ at $x$ along $z\in\tcone(x)$.
Given that $h$ is strictly convex, we have $\breg(\base,x) \geq 0$ and $x(t)\to\base$ whenever $\breg(\base,x(t))\to0$;
hence, the convergence of $x(t)$ to $\base$ can be checked by means of the associated divergence $\breg(\base,x(t))$.

Notwithstanding, if $h$ is not steep, it is often impossible to obtain information about $\breg(\base,x(t))$ from \eqref{eq:MD} if $x(t)$ is not interior.%
\footnote{%
To understand this, consider the case where $\feas=[0,1]$ and $\mirror = \Eucl$, the Euclidean projector of \cref{ex:Eucl}.
If we take the objective $\obj(x) = x$ and start \eqref{eq:MD} at $y_{0} = a > 1$, then $x(t)$ would be stuck at $1$ for all $t\in[0,a-1]$.
The Bregman divergence would not be able to detect the evolution of $y(t)$ in this case (in tune with the fact that \eqref{eq:MD} cannot be recast as an autonomous dynamical system in terms of $x$ when $h$ is not steep);
for a detailed discussion, see \cite{MS16}.
}
To overcome this difficulty, we will instead employ the so-called \emph{Fenchel coupling}
\begin{equation}
\label{eq:Fenchel}
\fench(\base,y)
	= h(\base) + h^{\ast}(y) - \braket{y}{\base}
	\quad
	\text{for all $\base\in\feas$, $y\in\dual$},
\end{equation}
so named because it collects all terms of Fenchel's inequality.%
\footnote{For a related, trajectory-based variant of $\fench$, see also \cite[p.~444]{ABRT04}.}
This ``primal-dual'' divergence was first introduced in \cite{MS16,Mer17} and, as a consequence of Fenchel's inequality, it follows that $\fench(\base,y) \geq 0$ with equality if and only if $\base = \mirror(y)$.

The following proposition (taken from \cite{Mer17}) links the Fenchel coupling with the Bregman divergence and the underlying norm:

\begin{proposition}
\label{prop:Fenchel}
Let $h$ be a $K$-strongly convex regularizer on $\feas$.
Then, for all $\base\in\feas$ and all $y,y'\in\dual$, we have:
\begin{subequations}
\label{eq:Fench-properties}
\begin{alignat}{2}
\label{eq:Fench-Bregman}
&a)
	\;\;
	\fench(\base,y)
	&&\geq \breg(\base,\mirror(y))
	\;
	\text{with equality whenever $\mirror(y)\in\intfeas$}.
	\hspace{4em}
	\\[2pt]
\label{eq:Fench-norm}
&b)
	\;\;
	\fench(\base,y)
	&&\geq \tfrac{1}{2} K \, \norm{\mirror(y) - \base}^{2}.
	\\[2pt]
\label{eq:Fench-diff}
&c)
	\;\;
	\fench(\base,y')
	&&\leq \fench(\base,y) + \braket{y' - y}{\mirror(y) - \base} + \tfrac{1}{2K} \dnorm{y'-y}^{2}.
\end{alignat}
\end{subequations}
\end{proposition}

\begin{proof}
See \cite[Proposition 4.3]{Mer17}.
\end{proof}

An immediate consequence of \eqref{eq:Fench-norm} is that $\mirror(y_{n})\to0$ for every sequence $(y_{n})_{n=0}^{\infty}$ in $\dual$ such that $\fench(\base,y_{n}) \to 0$.
As a result, the convergence of $x(t) = \mirror(y(t))$ to $\base\in\feas$ may be checked by showing that $\fench(\base,y(t))\to0$.
For technical reasons, it will be convenient to assume that the converse also holds, i.e.
\begin{equation*}
\label{eq:Fench-reg}
\tag{$\mathbf{H}_{2}$}
\fench(\base,y_{n}) \to 0
	\quad
	\text{whenever}
	\quad
	\mirror(y_{n}) \to \base.
\end{equation*}
When $h$ is steep, combining \cref{prop:mirror,prop:Fenchel} gives $\fench(\base,y) = \breg(\base,\mirror(y))$ for all $y\in\dual$,%
\footnote{To be clear, \cref{prop:Fenchel} guarantees that $\fench(\base,y) = \breg(\base,\mirror(y))$ whenever $\mirror(y)$ is interior.
The statement for steep $h$ is sharper because it states that $\fench(\base,y) = \breg(\base,\mirror(y))$ \emph{for all} $y$.
This is a consequence of the fact that $\im\mirror = \intfeas$ for steep $h$, hence the need to invoke \cref{prop:mirror}.}
so \eqref{eq:Fench-reg} boils down to the requirement
\begin{equation}
\breg(\base,x_{n}) \to 0
	\quad
	\text{whenever $x_{n}\to\base$}.
\end{equation}
This so-called ``reciprocity condition'' is well known in the theory of Bregman functions \cite{CT93,Kiw97b,ABB04} and, essentially, it means that the sublevel sets of $\breg(\base,\cdot)$ are neighborhoods of $\base$ in $\feas$.
Hypothesis \eqref{eq:Fench-reg} instead posits that the \emph{images} of the sublevel sets of $\fench(\base,\cdot)$ under $\mirror$ are neighborhoods of $\base$ in $\feas$, so \eqref{eq:Fench-reg} may be seen as a ``primal-dual'' variant of Bregman reciprocity.

It is easy to verify that \cref{ex:Eucl,ex:logit,ex:matrix} all satisfy \eqref{eq:Fench-reg}.
For an in-depth discussion of the geometric implications of Bregman reciprocity, the reader is referred to \cite{Kiw97b}.

\subsection{Deterministic analysis}
\label{sec:results-det}

Together with \cref{prop:mirror}, the Lipschitz continuity hypothesis \eqref{eq:Lipschitz} implies that the driving vector field $\payv(\mirror(\temp y))$ of \eqref{eq:MD} is itself Lipschitz continuous in $y$.
Hence, by standard results in the theory of differential equations, \eqref{eq:MD} is \emph{well-posed}, i.e. it admits a unique global solution for every initial condition $y_{0}\in\dual$ \cite[Chap.~V]{Rob95}.
With this in mind, we have:

\begin{theorem}
\label{thm:conv-det}
Assume \eqref{eq:Lipschitz} holds and let $x(t) = \mirror(\temp y(t))$ be a solution of \eqref{eq:MD} initialized at $y_{0}\in\dual$.
\begin{enumerate}
[\indent\upshape(1)]
\addtolength{\itemsep}{2pt}
\item
If $\obj_{\min}(t) = \min_{0\leq s\leq t} \obj(x(s))$ and $\bar\obj(t) = t^{-1} \int_{0}^{t} \obj(x(s))$ respectively denote the minimum and mean value of $\obj$ under \eqref{eq:MD}, we have
\begin{flalign}
\label{eq:rate-det}
&\obj_{\min}(t) - \min\obj
	\leq \bar\obj(t) - \min\obj
	= \bigoh(1/t).
	\\[2pt]
\intertext{In particular, if \eqref{eq:MD} is initialized at $y_{0} = 0$, we have}
\label{eq:rate-det0}
&\obj_{\min}(t)
	\leq \bar\obj(t)
	\leq \min\obj + \depth/t,
\end{flalign}
where $\depth = \max\setdef{h(x') - h(x)}{x,x'\in\feas}$.

\item
If \eqref{eq:Fench-reg} also holds, $x(t)$ converges to some $\sol\in\argmin\obj$ \textpar{possibly depending on $y_{0}$}.
\end{enumerate}
\end{theorem}

\cref{thm:conv-det} is a strong convergence result guaranteeing global trajectory convergence to a solution of \eqref{eq:program} and an $\bigoh(1/t)$ value convergence rate for the averaged process $\bar x(t) = t^{-1} \int_{0}^{t} x(s) \dd s$ (by Jensen's inequality).
In \cref{app:det}, we provide a Lyapunov-based proof leveraging the fact that the ``$\temp$-deflated'' Fenchel coupling
\begin{equation}
\label{eq:Fench-temp}
\lyap(t)
	= \temp^{-1} \fench(\sol,\temp y(t)),
\end{equation}
is nondecreasing along the solution orbits of \eqref{eq:MD} for all $\sol\in\argmin\obj$.

The first part of the theorem is well known and essentially dates back to the original work of Nemirovski and Yudin \cite{NY83}.
As for the trajectory convergence properties of \eqref{eq:MD}, \cite{BT03,ABB04} provide a proof for a \acl{HR} gradient system which is formally equivalent to \eqref{eq:MD} when $h$ is steep (for a detailed discussion, see \cref{sec:discussion});
\cite{ABRT04} also deals with the singular Riemannian case (corresponding to nonsteep $h$), but requires that $\feas$ be polyhedral.
Finally, \cite{ABB04,Kri16} also provide an $\bigoh(1/t)$ value convergence rate for $x(t)$;
under \eqref{eq:Fench-reg}, Part (ii) of \cref{thm:conv-det} narrows this convergence down to a \emph{point} $\sol\in\argmin\obj$ (instead of the \emph{set} $\argmin\obj$).%
\footnote{If $\mirror$ is smooth (as opposed to Lipschitz), a simple differentiation shows that $\obj(x(t))$ is nonincreasing in $t$.
In this case, an $\bigoh(1/t)$ convergence for $\obj(x(t))$ follows readily from an $\bigoh(1/t)$ upper bound on $\bar\obj(t)$ by noting that $\obj(x(t)) = t^{-1} \int_{0}^{t} \obj(x(t)) \dd s \leq t^{-1} \int_{0}^{t} \obj(x(s)) \dd s = \bar\obj(t)$.}

Building on this basic deterministic result, our aim in the rest of this paper will be to explore how the strong convergence properties of \eqref{eq:MD} are affected if the gradient input of \eqref{eq:MD} is contaminated by noise.

\section{Mirror descent with noisy gradient input}
\label{sec:noisyMD}

To account for noise and measurement errors in \eqref{eq:MD}, our starting point will be the random disturbance model
\begin{equation}
\label{eq:noise-Langevin}
\dot y(t)
	= \payv(x(t)) + \epsilon(t),
\end{equation}
where $\epsilon(t)$ is a random function of time representing the noise in the gradient input $\payv(x(t))$ at each instance $t\geq0$.
To write the Langevin equation \eqref{eq:noise-Langevin} as a formal \acl{SDE}, let $(\Omega,\filter,\{\filter_{t}\}_{t\geq0},\prob)$ be a filtered probability space,\footnote{We tacitly assume here that $\filter_{t}$ satisfies the usual conditions, i.e. it is complete ($\filter_{0}$ contains all $\prob$-null sets) and right-continuous ($\filter_{t} = \intersect_{s>t} \filter_{s}$).}
and consider the \acl{SMD} dynamics
\begin{equation}
\label{eq:SMD}
\tag{SMD}
\begin{aligned}
dY
	&= \payv(X) \dd t + d\noise,
	\\
X
	&= \mirror(\temp Y),
\end{aligned}
\end{equation}
where $\noise(t) = (\noise_{1}(t),\dotsc,\noise_{n}(t))$ is a continuous $\filter_{t}$-adapted Itô martingale.
More precisely, we assume throughout that $\noise(t)$ is of the general form
\begin{equation}
\label{eq:noise}
d\noise_{i}(t)
	= \sum_{k=1}^{m} \noisevol_{ik}(X(t),t) \dd W_{k}(t),
	\quad
	i=1,\dotsc,n,
\end{equation}
where:
\begin{enumerate}
\item
$W = (W_{1},\dotsc,W_{m})$ is an adapted $m$-dimensional Wiener process.%
\footnote{It is possible to consider even more general continuous semimartingale error terms here
, but the presentation would become much more complicated.}
\item
The $n\times m$ \emph{volatility matrix} $\noisevol_{ik}\from\feas\times\R_{+} \to \R$ of $\noise(t)$ is assumed measurable, bounded, and Lipschitz continuous in the first argument.
More formally, we posit that
\begin{equation}
\label{eq:noise-reg}
\tag{$\mathbf{H}_{3}$}
\begin{aligned}
\sup\nolimits_{x,t} \abs{\noisevol_{ik}(x,t)}
	&< \infty,
	\\
\abs{\noisevol_{ik}(x',t) - \noisevol_{ik}(x,t)}
	&\leq \ell \, \norm{x' - x}.
\end{aligned}
\end{equation}
for some
$\ell>0$
and for all $x,x'\in\feas$, $t\geq0$.
\end{enumerate}

The most straightforward case for the noise is when $m=n$ and $\noise(t) = \noisevol W(t)$ for constant $\noisevol$.
This case corresponds to \acs{iid} increments that are uncorrelated across different components and that do not depend on $t$ or $X(t)$.
However, these independence assumptions are not always realistic:
in \cref{sec:traffic}, we discuss an important example with nontrivial correlations that arise in the study of traffic networks, and which necessitate the more general treatment above.

More concretely, the correlation structure of the noise process $\noise$ can be captured by the \emph{quadratic covariation process} $[\noise(t),\noise(t)]$,%
\footnote{Recall here that the covariation of two processes $X$ and $Y$ is defined as $[X(t),Y(t)] = \lim_{\abs{\Pi}\to0} \sum_{1\leq j \leq k} (X(t_{j}) - X(t_{j-1})) (Y(t_{j}) - Y(t_{j-1}))$, where the limit is taken over all partitions $\Pi = \{t_{0} = 0 < t_{1} < \dotsi < t_{k} = t\}$ of $[0,t]$ with mesh $\abs{\Pi} \equiv \max_{j} \abs{t_{j} - t_{j-1}} \to 0$ \cite{KS98}.}
given here by the \acs{SDE}
\begin{flalign}
\label{eq:variation}
d[\noise_{i}(t),\noise_{j}(t)]
	&= \sum_{k,\ell=1}^{m} \noisevol_{ik}(X(t),t) \noisevol_{j\ell}(X(t),t) \dd W_{k}(t) \cdot dW_{\ell}(t)
	\notag\\
	&= \sum_{k=1}^{m} \noisevol_{ik}(X(t),t) \noisevol_{jk}(X(t),t) \dd t
	= \covmat_{ij}(X(t),t) \dd t,
\end{flalign}
where $\covmat = \noisevol \noisevol^{\top}$ is the \emph{infinitesimal covariance matrix} of the process.
If $\covmat$ is not diagonal, the components of $\noise$ exhibit nontrivial correlations quantified by the nonzero off-diagonal elements of $\covmat$.
This also highlights the role of the underlying $m$-dimensional Wiener process $W(t)$ in \eqref{eq:SMD}:
if $m<n$, the induced disturbances are necessarily correlated;
if $m=n$ and $\noisevol$ is diagonal, the errors are independent across components;
and if $m>n$, the noise in each component may result from the aggregation of several, independent error sources.
Obviously, the precise statistics of the noise depend crucially on the application being considered,
so, for generality, we maintain an application-agnostic approach and we make no assumptions on the structure of $\covmat$.

For posterity, we also note here that the noise regularity hypothesis \eqref{eq:noise-reg} gives
\begin{equation}
\label{eq:noisevar}
\fnorm{\noisevol(x,t)}^{2}
	\leq \noisevar
	\quad
	\text{for some $\noisedev>0$ and all $x\in\feas$, $t\geq0$},
\end{equation}
where
\begin{equation}
\label{eq:Frobenius}
\fnorm{\noisevol}
	\equiv\sqrt{\trof{\noisevol\noisevol^{\top}}}
	= \sqrt{\trof{\covmat}}
\end{equation}
denotes the Frobenius norm of the $n\times m$ matrix $\noisevol$.
In what follows, it will be convenient to measure the magnitude of the noise affecting \eqref{eq:SMD} via $\noisedev$;%
\footnote{Note here that $\noisevar$ typically scales with the dimensionality of $\feas$ (for instance, if $\noise$ is a standard $n$-dimensional Wiener process).}
obviously, when $\noisedev=0$, we recover the noiseless, deterministic dynamics \eqref{eq:MD}.

Now, under the Lipschitz continuity hypothesis \eqref{eq:Lipschitz} and the noise regularity condition \eqref{eq:noise-reg},
standard results from the theory of \aclp{SDE} show that \eqref{eq:SMD} admits unique strong solutions that exist for all time (see e.g. Theorem 3.21 in \cite{ParRas14}).%
\footnote{The Lipschitz continuity of the drift and diffusion terms of \eqref{eq:SMD} is key in this regard.}
Specifically, for every (random) $\filter_{0}$-measurable initial condition $Y_{0}$ with $\exof{\dnorm{Y_{0}}^{2}} < \infty$, there exists an almost surely continuous stochastic process $Y(t)$ satisfying \eqref{eq:SMD} for all $t\geq0$ and such that $Y(0) = Y_{0}$.
Furthermore, up to redefinition on a $\prob$-null set, $Y(t)$ is the unique $\filter_{t}$-adapted process with these properties \cite[Theorem~3.4]{Kha12}.

For concreteness, we will focus only on non-random initial conditions of the form $Y(0) = y_{0}$ for a fixed $y_{0}\in\dual$.
In this case, the second moment condition $\exof{\dnorm{Y(0)}^{2}} < \infty$ is satisfied automatically, so we have:

\begin{proposition}
\label{prop:wp}
Assume \eqref{eq:Lipschitz} and \eqref{eq:noise-reg} hold.
Then, for all $y_{0}\in\dual$ and up to a $\prob$-null set, \eqref{eq:SMD} admits a unique strong solution $(Y(t))_{t\geq0}$ such that $Y(0) = y_{0}$.
\end{proposition}

In the rest of the paper, when we refer to a solution trajectory of \eqref{eq:SMD}, we will implicitly invoke the well-posedness result above without making an explicit reference to it.

\section{Convergence results}
\label{sec:results}

Despite the strong convergence properties of the deterministic dynamics \eqref{eq:MD}, the noise-contaminated dynamics \eqref{eq:SMD} may fail to converge, even in simple, one-dimensional problems.
For an elementary example, take $\obj(x) = x^{2}/2$ over $\feas = [-1,1]$, let $\noise(t) = W(t)$:
since the martingale part of \eqref{eq:SMD} does not vanish when $X(t) = 0$, it follows that $X(t)$ cannot converge to $\argmin\obj = \{0\}$ with positive probability \textendash\ and this, independently of the choice of mirror map $\mirror$.

In view of this nonconvergent example, our aim in the rest of this section will be to:
\begin{enumerate}
\item
Analyze the convergence properties of \eqref{eq:SMD} in the ``vanishing noise'' regime (\cref{sec:smallnoise}).
\item
Study the long-run concentration properties of $X(t)$ around interior minimizers (\cref{sec:invariant}).
\item
Identify classes of convex programs where $X(t)$ \emph{does} converge (\cref{sec:sharp}).
\item
Examine a convergent variant of \eqref{eq:SMD} with a decreasing sensitivity parameter (\cref{sec:transforms}).
\end{enumerate}



\begin{table}[tbp]
\centering
\renewcommand{\arraystretch}{1.3}
\small

\begin{tabular}{cll}
	&\sc{Hypothesis}
	&\sc{Statement}
	\\
\hline
\eqref{eq:Lipschitz}
	&Lipschitz gradients
	&$\payv(x)$ is Lipschitz continuous
	\\
\hline
\eqref{eq:Fench-reg}
	&Bregman reciprocity
	&$\fench(\base,y_{n})\to0$ whenever $\mirror(y_{n})\to\base$
	\\
\hline
\eqref{eq:noise-reg}
	&Noise regularity 
	&$\noisevol(x,t)$ is bounded and Lipschitz in $x$
	\\
\hline
\end{tabular}
\vspace{2ex}
\caption{Overview of the various hypotheses used in the paper.}
\label{tab:rates}
\vspace{-2ex}
\end{table}


\subsection{The vanishing noise regime}
\label{sec:smallnoise}

We begin with the case where the gradient input to \eqref{eq:SMD} becomes more accurate as measurements accrue over time \textendash\ for instance, as in applications to wireless communications where the accumulation of pilot signals allows users to better sense their channel over time \cite{MBNS17}.
In this ``vanishing noise'' limit, intuition suggests that $X(t)$ should asymptotically follow the dynamics \eqref{eq:MD}, and hence converge (in some sense) to $\argmin\obj$.

To make this intuition precise, we first show below that $\argmin\obj$ is \emph{recurrent} under $X(t)$, i.e. $X(t)$ visits any neighborhood of $\argmin\obj$ infinitely often:

\begin{proposition}
\label{prop:smallnoise}
Assume \labelcref{eq:Lipschitz,eq:noise-reg} hold, and let $X(t) = \mirror(\temp Y(t))$ be a solution of \eqref{eq:SMD}.
If $\lim_{t\to\infty} \sup_{x\in\feas} \fnorm{\noisevol(x,t)} = 0$, there exists a \textpar{random} sequence of times $t_{n}\uparrow\infty$ such that $X(t_{n})\to\argmin\obj$ \as.
\end{proposition}

As in the noiseless case (and much of the analysis to follow), the proof of \cref{prop:smallnoise} hinges on the ``$\temp$-deflated'' Fenchel coupling
\begin{equation}
\lyap(t)
	= \temp^{-1} \fench(\sol,\temp Y(t)),
\end{equation}
which satisfies the (stochastic) Lyapunov-like property
\begin{alignat}{2}
\label{eq:dFench-stoch-const}
\lyap(t) - \lyap(0)
	&\leq \int_{0}^{t} \braket{\payv(X(s))}{X(s) - \sol} \dd s
	&\qquad&
	\text{(drift)}
	\notag\\
	&+ \frac{\temp}{2K} \int_{0}^{t} \trof{\covmat(X(s),s)} \dd s
	&\qquad&
	\text{(Itô correction)}
	\notag\\
	&+ \sum_{i=1}^{n} \int_{t_{0}}^{t} (X_{i}(s) - \sol_{i}) \dd Z_{i}(s)
	&\qquad&
	\text{(martingale noise)}
\end{alignat}
Arguing by contradiction, if $X(t)$ remained a bounded distance away from $\argmin\obj$, the drift term in \eqref{eq:dFench-stoch-const} would decrease linearly in $t$ for all $\sol\in\argmin\obj$ (by convexity).
Since the Itô correction and martingale noise terms grow sublinearly in $t$ (by the vanishing noise assumption and the law of large numbers respectively), this would give $\lyap(t)\to-\infty$, contradicting the fact that $\lyap(t)\geq0$.

Of course, \cref{prop:smallnoise} is considerably weaker than its deterministic counterpart (\cref{thm:conv-det}), because it does not even imply that $X(t)\to\argmin\obj$ with positive probability.
Nonetheless, by slightly strengthening the vanishing noise requirement $\sup_{x}\fnorm{\noisevol(x,t)}\to0$, we obtain that $X(t)\to\argmin\obj$
\emph{with probability $1$:}

\begin{theorem}
\label{thm:smallnoise}
Assume \labelcref{eq:Lipschitz,eq:Fench-reg,eq:noise-reg} hold and let $X(t) = \mirror(\temp Y(t))$ be a solution of \eqref{eq:SMD}.
If $\sup_{x\in\feas} \fnorm{\noisevol(x,t)} = o(1/\sqrt{\log t})$,
we have $X(t)\to\argmin\obj$ \as.
\end{theorem}

The key challenge in obtaining this a.s. convergence result is that, even if we ignored the martingale term in \eqref{eq:dFench-stoch-const}, it is quite difficult to balance the drift (helpful) and Itô correction (antagonistic) terms.
Thus, in lieu of a direct Lyapunov approach, we will show that $X(t)$ ``tracks'' the deterministic dynamics \eqref{eq:MD} in a certain, precise sense (see below), and then leverage the convergence properties of \eqref{eq:MD} to deduce that $X(t)\to\argmin\obj$.

To quantify what ``tracking'' means in this context, we use the seminal notion of an \acdef{APT} due to Benaïm and Hirsch \cite{BH96,Ben99}:

\begin{definition}
\label{def:APT}
Let $(Y(t))_{t\geq0}$ be a continuous curve in $\dual$ and let $\semiflow_{t}\from\dual\to\dual$, $t\geq0$, be the semiflow of \eqref{eq:MD} on $\dual$
\textpar{i.e. $(\semiflow_{t}(y))_{t\geq0}$ denotes the solution orbit of \eqref{eq:MD} that starts at $y\in\dual$}.
Then, $Y$ is an \acdef{APT} of $\semiflow$ if
\begin{equation}
\label{eq:APT}
\lim_{t\to\infty} \sup_{0\leq h \leq T} \dnorm{Y(t+h) - \semiflow_{h}(Y(t))}
	= 0
	\quad
	\text{for all $T>0$}.
\end{equation}
\end{definition}

Heuristically, an \ac{APT} of \eqref{eq:MD} asymptotically follows the induced semiflow $\semiflow$ with arbitrary accuracy over windows of arbitrary length.
Nonetheless, this ``fixed horizon'' property does not suffice to establish the convergence of an \ac{APT} to $\argmin\obj$, despite the strong convergence properties of \eqref{eq:MD}.
On that account, the basic steps of our proof are as follows:
\begin{enumerate}
[\itshape i\hspace*{1pt}\upshape)]
\item
Using the analysis of \cite{BH96}, we show that the stated decay assumption for $\noisevol(x,t)$ implies that solutions of \eqref{eq:SMD} are \acp{APT} of \eqref{eq:MD}.
\item
By \cref{prop:smallnoise}, $\argmin\obj$ is recurrent under \eqref{eq:SMD}, so solutions of \eqref{eq:SMD} cannot stray too far from $\argmin\obj$ in the long run.
\item
Once a solution of \eqref{eq:SMD} gets close enough to $\argmin\obj$, the \ac{APT} property means that it becomes trapped in its vicinity and eventually converges to it.
\end{enumerate}

We make all this precise in \cref{app:smallnoise} where we prove \cref{prop:smallnoise} and \cref{thm:smallnoise}.

\subsection{Long-run concentration around solution points}
\label{sec:invariant}

Beyond the vanishing noise regime, the simple example $\obj(x) = x^{2}/2$ with $\noise(t) = W(t)$ shows that $X(t)$ may fluctuate around $\argmin\obj$ in perpetuity if the noise is persistent.
As such, our goal in what follows will be to analyze the long-run concentration properties of \eqref{eq:SMD} and to determine
the domain that $X(t)$ occupies with high probability in the long run.

For reasons that will become clear shortly, we focus on strongly convex problems that admit a (necessarily unique) interior solution $\sol\in\intfeas$.
More concretely, this means that there exists some $\strong>0$ (related to the convexity of the problem) such that
\begin{equation}
\label{eq:strong}
\obj(x) - \obj(\sol)
	\geq \tfrac{1}{2} \strong \norm{x - \sol}^{2}
	\quad
	\text{for all $x\in\feas$}.
\end{equation}
Our first result in this case is as follows:

\begin{proposition} 
\label{prop:hitting}
Assume \labelcref{eq:Lipschitz,eq:noise-reg} hold, and let $\obj$ be an $\strong$-strongly convex function with an interior minimizer $\sol\in\intfeas$.
If $X(t) = \mirror(\temp Y(t))$ is a solution of \eqref{eq:SMD} initialized at $y_{0}\in\dual$, we have
\begin{equation}
\label{eq:MSE-mean}
\exof*{\frac{1}{t} \int_{0}^{t} \norm{X(s) - \sol}^{2} \dd s}
	\leq \frac{2\fench(\sol,\temp y_{0})}{\temp\strong t} + \frac{\temp\noisevar}{\strong K}.
\end{equation}
Moreover, if $\tau_{\delta} = \inf\setdef{t>0}{\norm{X(t) - \sol} \leq \delta}$ denotes the first time at which $X(t)$ gets within $\delta>0$ of $\sol$, we also have
\begin{equation}
\label{eq:hitting-bound}
\exof{\tau_{\delta}}
	\leq \frac{2 K \fench(\sol,\temp y_{0})}{\temp \strong K \delta^{2} - \temp^{2}\noisevar},
\end{equation}
provided that $\temp < \strong K\delta^{2}/\noisevar$.
In particular, for $y_{0}=0$, we have the optimized bound
\begin{equation}
\label{eq:hitting-bound-opt}
\exof{\tau_{\delta}}
	\leq \frac{8\depth\noisevar}{\strong^{2}K\delta^{4}},
\end{equation}
achieved for $\temp = \strong K\delta^{2}/(2\noisevar)$.
\end{proposition}

\begin{remark}
In the above, the constant $\strong$ has to do with the objective function $\obj$ and the feasible region $\feas$, while $K$ and $\depth$ are linked to the mirror map $\mirror$ (and, of course, also $\feas$).
The optimizer has no control over the former, but if its value can be estimated and the geometry of $\feas$ is relatively simple, the latter can be finetuned further to sharpen the above bounds.
\end{remark}

\begin{remark}
For a value-based analogue of \eqref{eq:MSE-mean} when $h$ is steep, see \cite[Prop.~4]{RB13}.
\end{remark}

\cref{prop:hitting} (proved in \cref{app:invariant}) provides a basic estimate of the long-run concentration of $X(t)$ around $\sol$, and also highlights the role of $\strong$ and $\noisedev$.
Specifically, \eqref{eq:hitting-bound-opt} shows that $X(t)$ hits a $\delta$-neighborhood of $\sol$ in time which is $\bigoh(1/\delta^{4})$ on average;
what's more, the multiplicative constant in this bound increases with the noise level in \eqref{eq:SMD} and decreases with the sharpness of the minimum point $\sol$ (as quantified by the strong convexity constant $\strong$ of $\obj$).

To obtain finer information regarding the concentration of $X(t)$ around $\sol$, we need to consider its occupation measure:

\begin{definition}
\label{def:occ}
The \emph{occupation measure} of $X$ at time $t\geq0$ is given by
\begin{equation}
\label{eq:occ}
\occ_{t}(A)
	= \frac{1}{t} \int_{0}^{t} \one(X(s) \in A) \dd s
	\quad
	\text{for every Borel $A\subseteq\feas$}.
\end{equation}
\end{definition}


\begin{figure}[tbp]
\centering
\subfigure{\includegraphics[width=.485\textwidth]{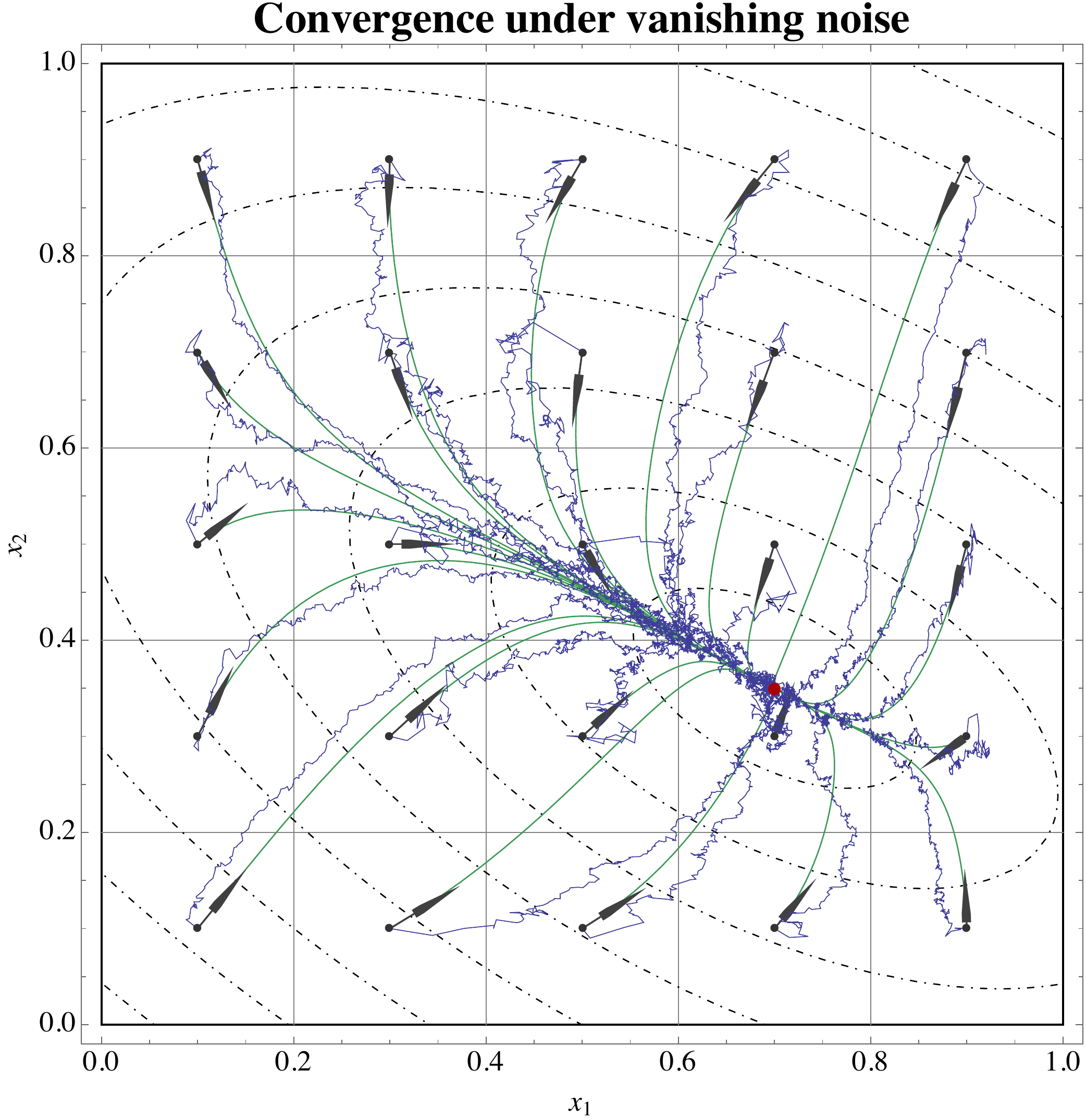}}
\hfill
\subfigure{\includegraphics[width=.485\textwidth]{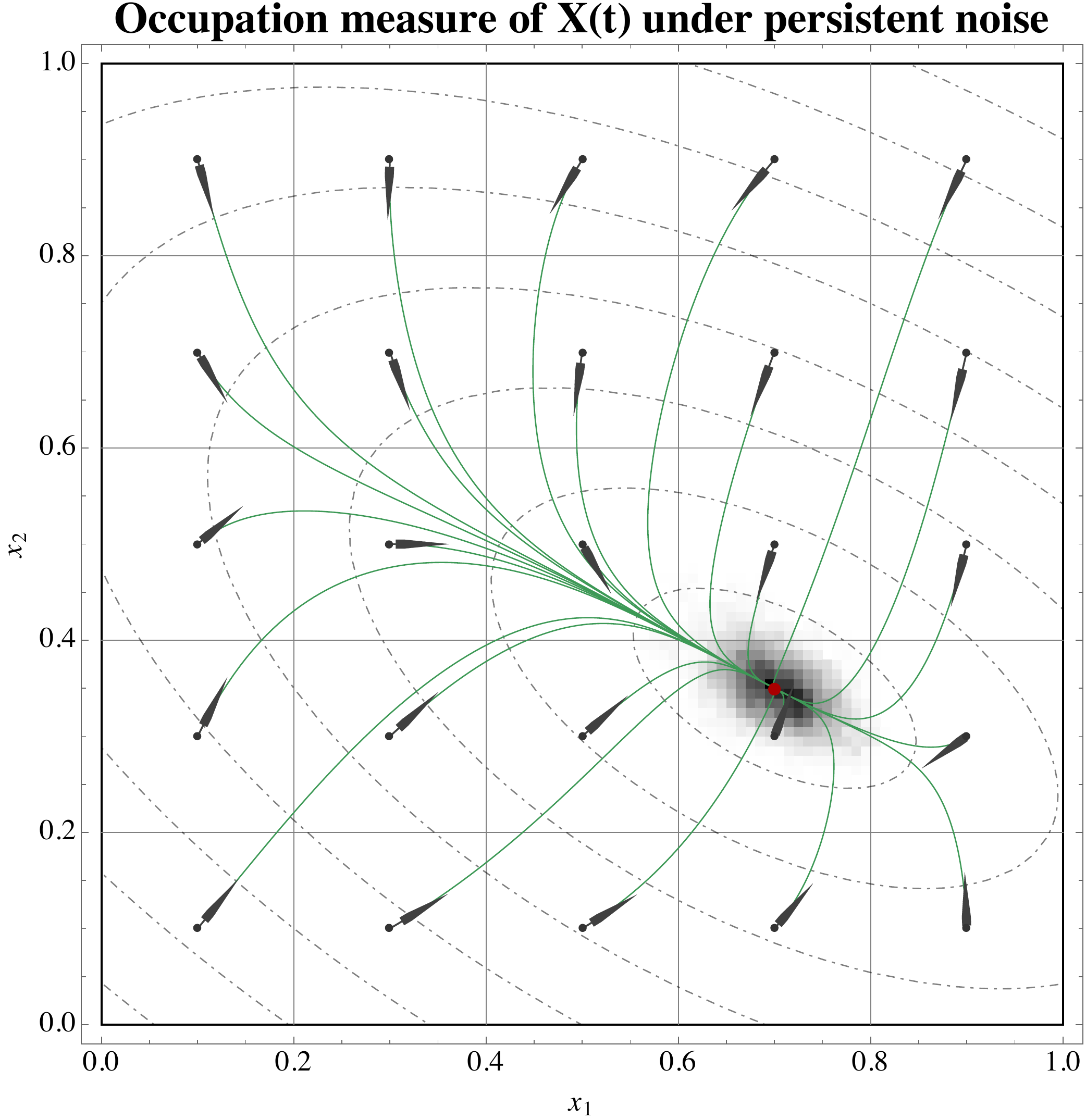}}%
\vspace{-1ex}
\caption{%
Numerical illustration of \eqref{eq:SMD} with $\mirror(y) = e^{y}/(1 + e^{y})$.
The dashed contours represent the level sets of $\obj$ over $\feas = [0,1]^{2}$,
and
the flowlines indicate the flow of \eqref{eq:MD}.
In the first figure, we exhibit the convergence of \eqref{eq:SMD} to $\argmin\obj$ when the volatility of the noise decays as $\Theta(1/\log t)$.
In the second, we estimate the long-run occupation measure of $X$:
darker shades of gray correspond to higher probabilities of observing $X$ in a given region.}
\label{fig:numerics}
\vspace{-1ex}
\end{figure}


In words, $\occ_{t}(A)$ is the fraction of time that $X$ spends in $A$ up to time $t$.
As such, the asymptotic concentration of $X$ around $\sol$ can be estimated by the quantity $\occ_{t}(\ball_{\delta})$, where
\begin{equation}
\label{eq:ball}
\ball_{\delta}
	\equiv \ball_{\delta}(\sol)
	= \setdef{x\in\feas}{\norm{x - \sol} \leq \delta}
\end{equation}
is the intersection of a $\delta$-ball centered at $\sol$ with $\feas$.
We then have the following concentration result (for a numerical illustration, see \cref{fig:numerics}):

\begin{theorem}
\label{thm:invariant}
Assume \labelcref{eq:Lipschitz,eq:noise-reg} hold, and let $\obj$ be an $\strong$-strongly convex function admitting an interior minimizer $\sol\in\intfeas$.
Moreover, fix some $\delta>0$ and suppose that the infinitesimal covariance matrix $\covmat$ of \eqref{eq:SMD} is time-homogeneous and uniformly po\-si\-tive-de\-fi\-nite \textpar{i.e. $\covmat(x,t) \equiv \covmat(x) \mgeq \lambda I$ for some $\lambda>0$}.
If \eqref{eq:SMD} is run with $\temp < \strong K\delta^{2}/\noisevar$,
then
\begin{equation}
\label{eq:inv-bound}
\occ_{t}(\ball_{\delta})
	\gtrsim 1 - \frac{\temp\noisevar}{\strong K\delta^{2}}
	\quad
	\text{for sufficiently large $t$ \as}.
\end{equation}
\end{theorem}

\begin{corollary}
\label{cor:conc}
Fix some tolerance $\eps>0$.
If \eqref{eq:SMD} is run with assumptions as above and $\temp \leq \eps \strong K\delta^{2}/\noisevar$, we have $\occ_{t}(\ball_{\delta}) \geq 1-\eps$ for all sufficiently large $t$ \as.
\end{corollary}

\begin{remark}
Since $\covmat = \noisevol\noisevol^{\top}$, it follows that $\covmat$ is nonnegative-definite by default.
The stronger assumption $\covmat \mgeq \lambda I$ essentially posits that the volatility matrix $\noisevol$ of $\noise$ has $\rank(\noisevol) = n$, i.e. the components of $\noise$ are not completely correlated.
For instance, this condition is trivially satisfied in the baseline case where $\noise$ is a Wiener process in $\R^{n}$.
\end{remark}

\begin{remark}
It is also worth noting that the bound \eqref{eq:inv-bound} only depends on the mirror map $\mirror$ via its inverse Lipschitz constant $K$ (that is, the strong convexity constant of $h$).
\cref{eq:inv-bound} suggests that $K$ should be taken as large as possible (to have $\occ_{t}(\ball_{\delta})\approx1$).
However, in so doing, the process $X(t)$ will initially spend a much larger amount of time near the prox-center $x_{c} \equiv \argmin h$ of $\feas$, so there is a trade-off between the sharpness of the asymptotic concentration of $X(t)$ near $\sol$ and the time it takes to attain this asymptotic regime.
\end{remark}

In a nutshell, \cref{thm:invariant} states that the concentration of $X(t)$ around $\sol$ may be arbitrarily sharp if $\temp$ is taken small enough.
Indeed, for $\temp<\strong K\delta^{2}/\noisevar$, \cref{prop:hitting} shows that $\ball_{\delta}$ is \emph{recurrent}, i.e. $\probof{X(t) \in \ball_{\delta} \; \text{for some} \; t\geq0} = 1$ for every initial condition $y_{0}\in\dual$.
Relegating the (fairly intricate) details to \cref{app:invariant}, it can be shown that the stated assumptions guarantee the existence of a unique invariant distribution $\nu$ for the dual process $Y(t)$.
The pushforward of $\nu$ to $\feas$ is precisely the limit of the occupation measures $\occ_{t}$ of $X$ as $t\to\infty$, so \eqref{eq:inv-bound} follows by using the mean square bound \eqref{eq:MSE-mean} to estimate $\nu$.

We close this section by noting that the assumption that $\sol$ is interior is crucial in the statement of \cref{thm:invariant}.
As we shall see in the next section, if $\sol$ is a corner of $\feas$ (i.e. $\pcone(\sol)$ has nonempty interior), $Y(t)$ is \emph{transient} (not recurrent) and $X(t)$ \emph{converges} to $\sol$ (instead of fluctuating in a small neighborhood thereof).
Otherwise, when $\sol$ belongs to a nontrivial face of $\feas$, the dynamics \eqref{eq:SMD} exhibit a hybrid behavior:
$X(t)$ converges \as to the smallest face of $\feas$ that contains $\sol$ and fluctuates around $\sol$ along the relative interior of said face.
However, obtaining a precise result along these lines is fairly cumbersome, so we omit this analysis.

\subsection{Sharp solutions and linear programming}
\label{sec:sharp}

Consider now the elementary linear program
\begin{equation}
\label{eq:linear-simple}
\begin{aligned}
\textrm{minimize}
	&\quad
	1-x,
	\\
\textrm{subject to}
	&\quad
	0 \leq x \leq 1.
\end{aligned}
\end{equation}
Taking for concreteness $\temp=1$, $h(x) = x\log x + (1-x) \log(1-x)$ and $\noise(t) = \noisevol W(t)$ with constant $\noisevol$, the dynamics \eqref{eq:SMD} become
\begin{equation}
\label{eq:SMD-linear}
\begin{aligned}
dY
	&= dt + \noisevol\dd W,
	\\
X
	&= e^{Y}/(1+e^{Y}),
\end{aligned}
\end{equation}
and, after integrating, we get $Y(t) = t + \noisevol W(t)$.
By a trivial stochastic estimate, this implies that $Y(t)\geq t/2$ for large $t$ (except possibly on a $\prob$-null set), so $\lim_{t\to\infty} X(t) = 1$ \as.
In other words, in the simple linear program \eqref{eq:linear-simple}, $X(t)$ converges to $\argmin\obj$ with probability $1$, no matter the level of the noise.

The reason behind this convergence (as opposed to the case of interior minimizers) is that the drift of \eqref{eq:SMD-linear} does not vanish when $X(t)$ approaches $\argmin\obj$, so it ends up dominating the martingale term $W(t)$.
A nonvanishing gradient is typical of (generic) linear programs, so one would optimistically expect comparable results to hold whenever \eqref{eq:program} can be locally approximated by a linear program.
Following Polyak \cite[Chapter~5.2]{Pol87}, we formalize this idea by focusing on convex programs with \emph{sharp} solutions:

\begin{definition}
\label{def:sharp}
We say that $\sol\in\feas$ is a \emph{$\sharp$-sharp} minimum point of $\obj$ if
\begin{equation}
\label{eq:sharp}
\obj(x)
	\geq \obj(\sol) + \sharp \norm{x - \sol}
	\quad
	\text{for some $\sharp>0$ and all $x\in\feas$}.
\end{equation}
\end{definition}

From \cref{def:sharp}, it is easy to see that a sharp minimum point is the unique minimizer of $\obj$ and it remains invariant under small perturbations of $\obj$ (assuming of course that such a minimizer exists in the first place).
On top of that, with $\obj$ assumed smooth,%
\footnote{\cref{def:sharp} is meaningful even if $\obj$ is not smooth, but we only treat smooth functions here.}
we also have the following geometric characterization:

\begin{lemma}
\label{lem:sharp}
$\sol\in\feas$ is a $\sharp$-sharp solution of \eqref{eq:program} if and only if
\begin{equation}
\label{eq:sharp-tangent}
\braket{\payv(\sol)}{z}
	\leq - \sharp\norm{z}
	\quad
	\text{for some $\sharp>0$ and for all $z\in\tcone(\sol)$}.
\end{equation}
\end{lemma}

\begin{proof}
The ``if'' part follows trivially by convexity.
For the ``only if'' part, let $z\in\tcone(\sol)$ and note that \eqref{eq:sharp} gives
\begin{equation}
\frac{\obj(\sol + tz) - \obj(\sol)}{t}
	\geq \sharp \norm{z}
	\quad
	\text{for all sufficiently small $t>0$}.
\end{equation}
Hence, taking the limit $t\to0^{+}$, we get $\braket{\nabla\obj(\sol)}{z} \geq \sharp\norm{z}$ and \eqref{eq:sharp-tangent} follows.
\end{proof}

A further consequence of \cref{lem:sharp} is that $\payv(\sol)\in\intr(\pcone(\sol))$, implying that sharp solutions of smooth convex programs can only occur at \emph{corners} of $\feas$ (that is, points whose polar cone has nonempty topological interior).
In this sense, sharp minimizers constitute the flip side of the interior-point analysis of the previous section, a contrast which is further reflected in the following a.s. convergence result:%

\begin{theorem}
\label{thm:sharp}
Assume \labelcref{eq:Lipschitz,eq:Fench-reg,eq:noise-reg} hold and suppose that $\obj$ admits a \textpar{necessarily unique} sharp minimum point $\sol$.
If \eqref{eq:SMD} is run with a sufficiently small sensitivity parameter $\temp$, $X(t)$ converges to $\sol$ \as;
in addition, if the mirror map $\mirror$ is surjective, this convergence occurs in finite time \as.
\end{theorem}


As an important special case, note that every solution of a (generic) linear program is sharp.%
\footnote{``Generic linear program'' means here that $\feas$ is a polytope, $\obj\from\feas\to\R$ is affine, and $\obj$ is constant only along the zero-dimensional faces of $\feas$ \cite{Pol87}.}
\cref{thm:sharp} then gives:

\begin{corollary}
\label{cor:linear}
If \eqref{eq:program} is a generic linear program and \eqref{eq:SMD} is run with Euclidean projections \textpar{cf.~\cref{ex:Eucl}} and small enough $\temp$, $X(t)$ converges to $\argmin\obj$ in finite time \as.
\end{corollary}

To gain some insight in the proof of \cref{thm:sharp}, note first that the driving vector field $\payv(x)$ of \eqref{eq:SMD} points towards $\sol$ for all $x\in\feas$ (by convexity).
Thanks to this basic property, almost every solution of \eqref{eq:SMD} visits any neighborhood of $\sol$ infinitely many times \as.
However, when $X(t)$ is near $\sol$, the sharpness of the solution ``traps'' $X(t)$ near $\sol$ and does not allow any overshoots (as in the interior case) because $\sol$ is a corner of $\feas$.
By a hitting time argument based on Girsanov's theorem, it is then possible to show that the dual process $Y(t)$ escapes to infinity along a direction contained in the polar cone $\pcone(\sol)$ of $\feas$ at $\sol$.
Then, the a.s. convergence of $X(t)$ to $\sol$ follows from a straightforward geometric argument.

We make all this precise in \cref{app:sharp}.

\subsection{Rectification}
\label{sec:transforms}

In this section, we examine a ``rectified'' variant of \eqref{eq:SMD} which is run with a decreasing sensitivity parameter and which takes into account all past information up to time $t$.
Specifically, motivated by \hyperref[thm:conv-det]{Theorem \ref*{thm:conv-det}(i)}, consider the transformed process
\begin{subequations}
\label{eq:X-var}
\begin{flalign}
\label{eq:X-avg}
\wilde X(t)
	&= \frac{1}{t} \int_{0}^{t} X(s) \dd s,
\intertext{or}
\label{eq:X-best}
\wilde X(t)
	&= X(s_{t})
	\quad
	\text{with $s_{t}\in\argmin\nolimits_{0\leq s \leq t} \obj(X(s))$},
\end{flalign}
\end{subequations}
corresponding respectively to the long-run average (also known as the ``ergodic average'' in optimization) and the ``best value'' of $X$ up to time $t$.

The results of \cite{RB13} and the analysis of \cref{sec:invariant} indicate that $\wilde X(t)$ is concentrated around interior solutions of $\feas$ (in the long run and in probability), provided that \eqref{eq:SMD} is run with sufficiently small $\temp$.
That said, in a black-box setting where knowledge about \eqref{eq:program} and the noise process $\noise(t)$ is not readily available, the choice of $\temp$ would essentially become a matter of trial and error.
Thus, a meaningful work-around would be to employ a variable sensitivity parameter $\temp\equiv\temp(t)$ which decreases to $0$ as $t\to\infty$.

Since $Y(t) = \bigoh(t)$ by the Lipschitz assumption \eqref{eq:Lipschitz}, $\temp(t)$ should not decrease to zero faster than $1/t$:
otherwise, $X(t) = \mirror(\temp(t) Y(t))$ would converge to the prox-center $x_{c} \equiv \argmin_{x\in\feas} h(x)$ of $\feas$ with probability $1$.
With this in mind, we make the following assumption throughout this section:
\begin{equation}
\label{eq:temp}
\tag{$\mathbf{H}_{4}$}
\text{$\temp(t)$ is Lipschitz continuous, nonincreasing, and $\lim\nolimits_{t\to\infty} t\temp(t) = \infty$}.
\end{equation}
Under this assumption, we have:

\begin{theorem}
\label{thm:conv-var}
Assume \labelcref{eq:Lipschitz,,eq:noise-reg,eq:temp} hold.
Then, the rectified process $\wilde X(t)$ enjoys the performance guarantees
\begin{flalign}
\label{eq:rate-as}
\obj(\wilde X(t))
	&\leq \min\obj
	+ \frac{\depth}{t \temp(t)}
	+ \frac{\noisevar}{2Kt} \int_{0}^{t} \temp(s) \dd s
	+ \bigoh(\sqrt{\log \log t/t})
	\quad
	\text{\textup(a.s.\textup)},
\intertext{and}
\label{eq:rate-mean}
\exof{\obj(\wilde X(t))}
	&\leq \min\obj
	+ \frac{\depth}{t\temp(t)}
	+ \frac{\noisevar}{2Kt} \int_{0}^{t} \temp(s) \dd s
	+ \bigoh(1/t),
\end{flalign}
where $\depth = \max\setdef{h(x') - h(x)}{x,x'\in\feas}$.
In particular, if $\lim_{t\to\infty} \temp(t) = 0$, we have $\wilde X(t) \to \argmin\obj$ \textup(a.s.\textup).
\end{theorem}


\begin{corollary}
\label{cor:rate-var}
Suppose that $\temp(t) \propto t^{-\beta}$ for some $\beta\in(0,1)$ and all $t\geq1$.
Then:
\begin{equation}
\label{eq:rate-var}
\obj(\wilde X(t)) - \min\obj
	= \begin{cases}
	\bigoh\left(t^{-\beta}\right)
		&\quad
		\text{if $0<\beta<\frac{1}{2}$},
		\\
	\bigoh\left(\sqrt{\log \log t /t}\right)
		&\quad
		\text{if $\beta=\frac{1}{2}$},
		\\
	\bigoh\left(t^{\beta-1}\right)
		&\quad
		\text{if $\frac{1}{2}<\beta<1$}.
	\end{cases}
\end{equation}
\end{corollary}

\begin{corollary}
\label{cor:rate-opt}
If $\temp(t) = \sqrt{\depth K/\noisevar} \, \min\{1,1/\sqrt{t}\}$, we have
\begin{equation}
\label{eq:rate-mean-opt}
\exof{\obj(\wilde X(t))}
	\leq \min\obj + 2 \sqrt{\depth \noisevar/(Kt)}.
\end{equation}
\end{corollary}

Compared to \eqref{eq:rate-det}, \cref{cor:rate-opt} indicates a drop in convergence speed from $\bigoh(1/t)$ to $\bigoh(1/\sqrt{t})$.
This is due to the Itô correction term $\noisevar/(2Kt) \int_{0}^{t}\temp(s) \dd s$ in \eqref{eq:rate-mean}:
balancing this second-order error against the noise-free bound $\depth/(t\temp(t))$ imposes a $\Theta(1/\sqrt{t})$ schedule for $\temp(t)$ \textendash\ otherwise, one term would be asymptotically slower than the other.
In this regard, \eqref{eq:rate-mean-opt} is reminiscent of the well-known $\bigoh(1/\sqrt{t})$ bounds derived in \cite[Section 2.3]{NJLS09} and \cite[Section~6]{Nes09} for the dual averaging method \eqref{eq:MD-discrete} in stochastic environments.
As discussed in \cite{KM17}, the drop in performance from $\bigoh(1/t)$ to $\bigoh(1/\sqrt{t})$ in the discrete-time case stems from the gap between continuous and discrete time:
specifically, the discretization of the continuous-time dynamics introduces a second-order Taylor term which slows down convergence.
In the case of \eqref{eq:SMD}, the second-order error that appears is not due to discretization, but to the (second-order) Itô correction which has a similar effect.

\section{Discussion}
\label{sec:discussion}

In this last section, we discuss some applications and extensions of our analysis so far.

\subsection{The traffic assignment problem: a case study}
\label{sec:traffic}

We begin with an application of our results to traffic assignment, a key problem in transportation and network science that concerns the optimal selection of paths between origins and destinations in traffic networks.
Referring to \cite{BMW56,BG92} for a detailed discussion, the core incarnation of the problem is as follows:
First, let $\graph = (\vertices,\edges)$ be a directed multi-graph with vertex set $\vertices$ and edge set $\edges$.
Assume further that there is an \ac{OD} pair $(\source,\sink)\in\vertices\times\vertices$ sending $\rate$ units of traffic from $\source$ to $\sink$ via a set of paths $\route\in\routes$ (that is, a set of simple edge chains joining $\source$ to $\sink$ in $\graph$ in the usual way).%
\footnote{The extension of the model to networks with multiple \ac{OD} pairs requires more elaborate notation, but is otherwise straightforward;
for an atomic, nonsplittable variant, see \cite{BM17}.}
The set of feasible \emph{routing flows} $\flow = (\flow_{\route})_{\route\in\routes}$ in the network is then defined as
\begin{equation}
\label{eq:flows}
\txs
\flows
	= \rate\,\simplex(\routes)
	= \setdef[\big]{(\flow_{\route})_{\route\in\routes}}{\text{$\flow_{\route}\geq0$ and $\sum_{\route\in\routes} \flow_{\route} = \rate$}}.
\end{equation}

Given a routing flow $\flow\in\flows$, the \emph{load} on edge $\edge\in\edges$ is $\load_{\edge} = \sum_{\route\ni\edge} \flow_{\route}$ and the \emph{delay} experienced by an infinitesimal traffic element traversing edge $\edge$ is $\cost_{\edge}(\load_{\edge})$, where $\cost_{\edge}\from\R_{+}\to\R_{+}$ is a nondecreasing convex \emph{cost function} (often a polynomial with positive coefficients).
Then, the delay along path $\route\in\routes$ is given by
\begin{equation}
\label{eq:cost-route}
\cost_{\route}(\flow)
	\equiv \sum_{\edge\in\route} \cost_{\edge}(\load_{\edge}),
\end{equation}
and the average delay in the network will be
\begin{equation}
\Cost(\flow)
	= \sum_{\route\in\routes} \flow_{\route} \cost_{\route}(\flow)
	= \sum_{\route\in\routes} \sum_{\edge\in\edges} \flow_{\route} \cost_{\edge}(\load_{\edge})
	= \sum_{\edge\in\edges} \load_{\edge} \cost_{\edge}(\load_{\edge}).
\end{equation}

In this setting,
solving the traffic assignment problem means finding a socially optimum routing flow $\sol\in\argmin_{\flow\in\flows}\Cost(\flow)$.
Assuming that the controlling \ac{OD} pair updates its routing flow at each $t\geq0$, \cref{thm:conv-det} shows that an optimum flow can be attained in an online manner by following the dynamics \eqref{eq:MD}.
More precisely, if we introduce the \emph{marginal cost}
\begin{equation}
\label{eq:cost-marginal}
\tilde\cost_{\edge}(\load_{\edge})
	= (\load_{\edge} \cost_{\edge}(\load_{\edge}))'
	= \cost_{\edge}(\load_{\edge}) + \load_{\edge} \cost_{\edge}'(\load_{\edge})
\end{equation}
and its path-based analogue $\tilde\cost_{\route}(\flow) = \sum_{\edge\in\route} \tilde\cost_{\edge}(\load_{\edge})$, a simple differentiation yields
\begin{equation}
\label{eq:cost-gradient}
\frac{\pd\Cost}{\pd \flow_{\route}}
	= \sum_{\edge\in\route} \tilde\cost_{\edge}(\load_{\edge})
	= \tilde\cost_{\route}(\flow).
\end{equation}
Thus, the dynamics \eqref{eq:MD} take the form
\begin{equation}
\label{eq:MD-traffic}
\begin{aligned}
\dot y_{\route}
	&= -\tilde\cost_{\route}(\flow),
	\\
\flow
	&= \mirror(\temp y),
\end{aligned}
\end{equation}
and, assuming $\cost$ and $h$ are sufficiently regular,%
\footnote{For instance, this is so if $\cost_{\edge}$ is polynomial and $h$ is the entropic regularizer of \cref{ex:logit}.}
\cref{thm:conv-det} shows that every solution $x(t)$ of \eqref{eq:MD-traffic} converges to an optimum routing flow $\sol\in\argmin\Cost$.

Now, if the marginal cost of each edge is only observable up to a random error, the scoring step of \eqref{eq:MD-traffic} takes the form
\begin{equation}
\label{eq:SMD-traffic}
\begin{aligned}
dY_{\route}
	&= - \sum_{\edge\in\route} \bracks{\tilde\cost_{\edge}(\load_{\edge}) \dd t + \noisevol_{\edge} \dd W_{\edge}}
	= - \tilde\cost_{\route}(X) \dd t + \dd\noise_{\route},
	\\
X
	&= \mirror(\temp Y),
\end{aligned}
\end{equation}
where $d\noise_{\route} = -\sum_{\edge\in\route} \noisevol_{\edge} \dd W_{\edge}$ and $\noisevol_{\edge}$ is assumed constant (for simplicity).
An easy calculation then shows that the infinitesimal covariance matrix $\covmat$ of $\noise$ is given by
\begin{flalign}
\covmat_{\route\routealt}
	= \sum_{\edge,\edgealt\in\edges} \noisevol_{\edge} \noisevol_{\edgealt} \delta_{\edge\edgealt} \one(\edge\in\route) \one(\edgealt\in\routealt)
	= \sum_{\edge\in\route\cap\routealt} \noisevol_{\edge}^{2}
\end{flalign}
i.e. stochastic fluctuations across two different paths $\route,\routealt\in\routes$ are correlated over their common edges.
This provides an important example where different components of the noise process $\noise$ are inherently correlated \textendash\ here, due to the underlying graph $\graph$.


\begin{figure}[t]
\centering
\footnotesize
\subfigure
[Continental US network topology.]
{\includegraphics[width=.48\textwidth]{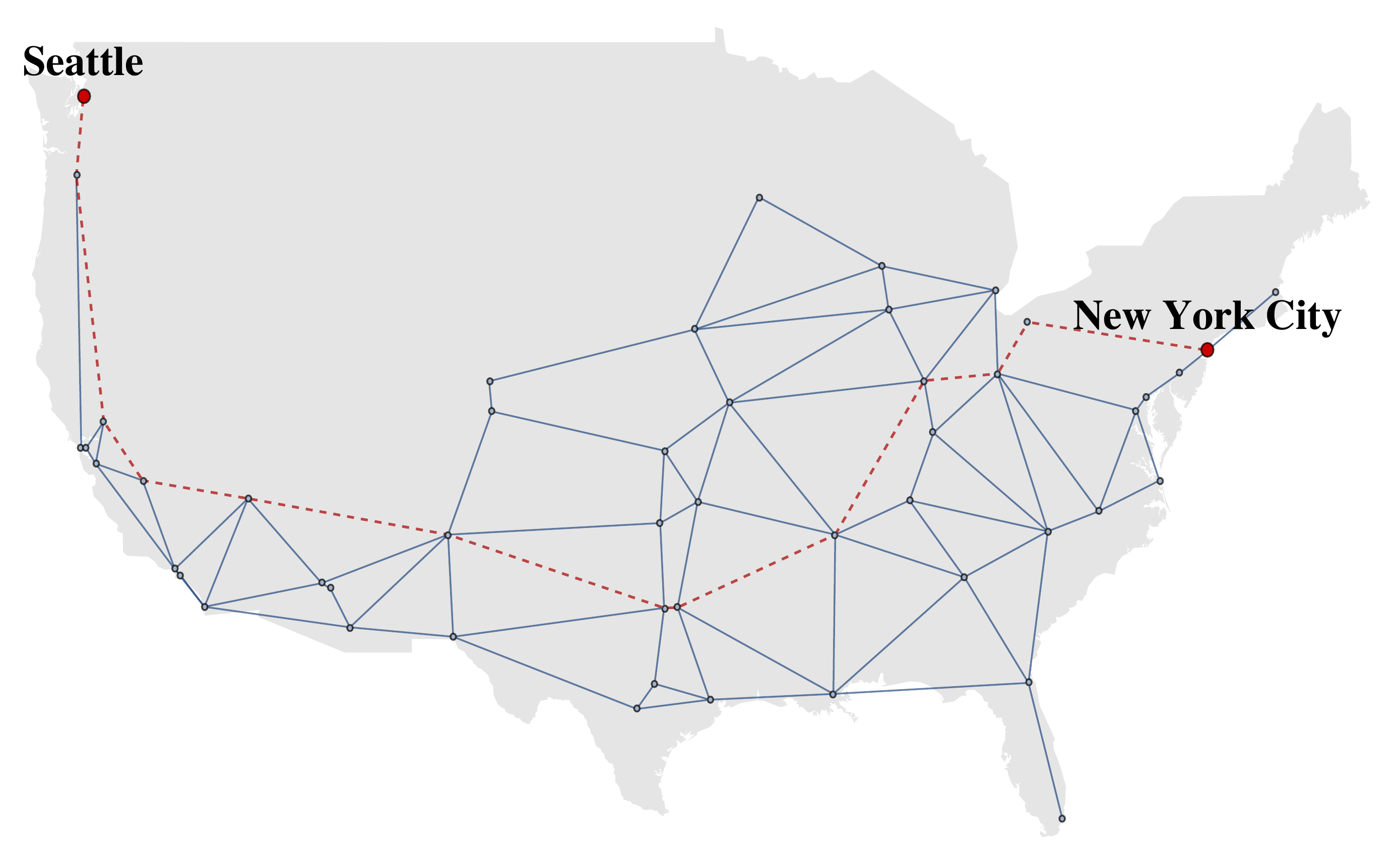}
\label{fig:network-US}}
\hfill
\subfigure
[Convergence of the ergodic average $\bar X(t)$.]
{\includegraphics[width=.48\textwidth]{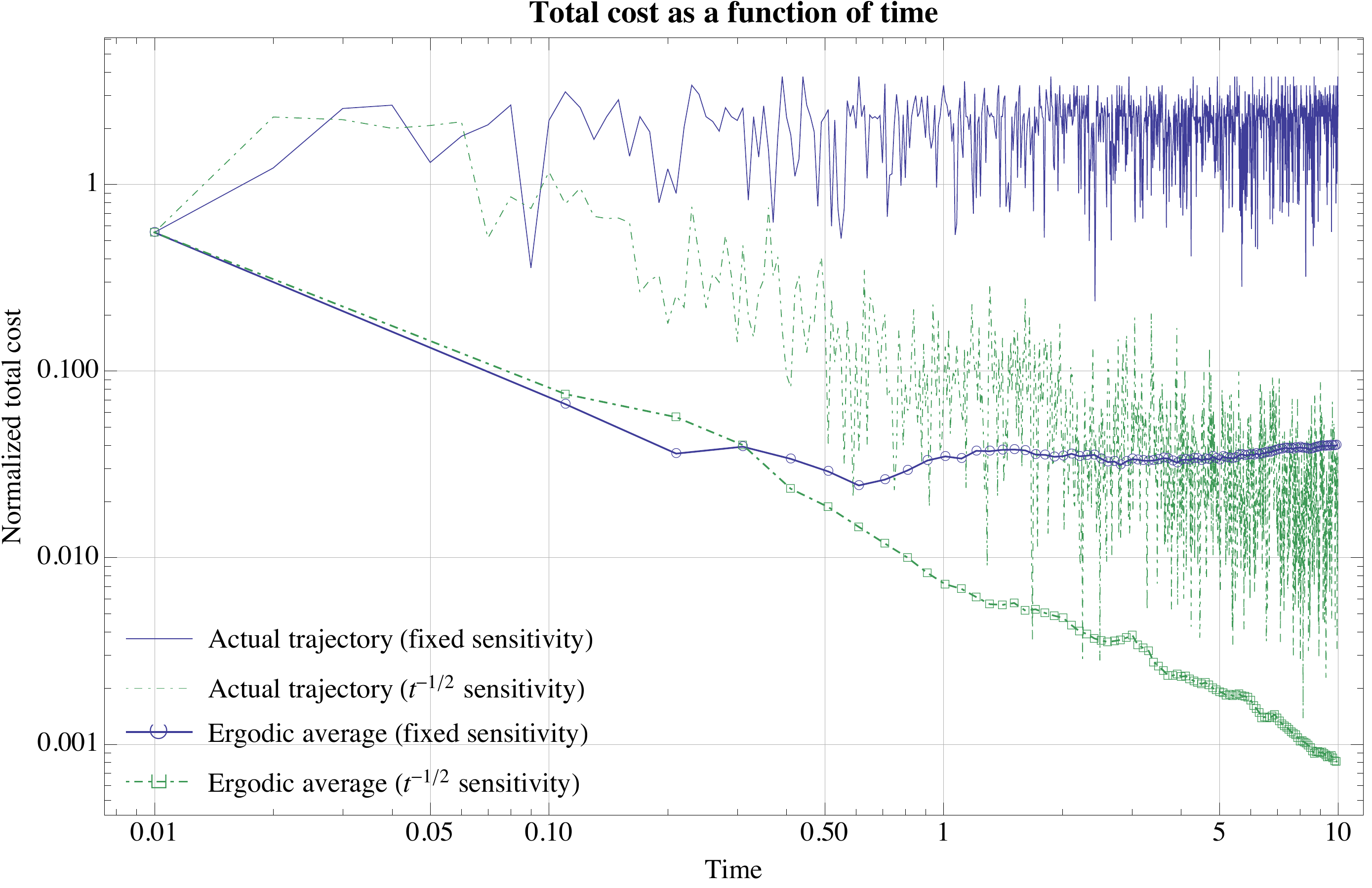}
\label{fig:network-convergence}}
\caption{Evolution of the dynamics \eqref{eq:SMD} in the traffic assignment problem.
\cref{fig:network-US} shows the underlying fiber network for the $50$ largest continental US cities.
In \cref{fig:network-convergence},
we provide a log-log plot of the normalized total cost $\Cost_{0}(x) = \Cost(x) - \min\Cost$ under \eqref{eq:SMD} with logit choice (\cref{ex:logit}).
When run with a fixed sensitivity, $X(t)$ meanders around without converging (solid blue line) and even the time-averaged process $\bar X(t) = t^{-1}\int_{0}^{t} X(s) \dd s$ fails to converge (blue line with circle markers).
If run with a $t^{-1/2}$ sensitivity schedule, $X(t)$ gets closer to the optimum (dashed green line) and its time-average follows a power law (dashed green line with square markers).%
}
\label{fig:network}
\end{figure}


\cref{fig:network} shows the evolution of the dynamics \eqref{eq:MD-traffic} in a data network consisting of the $50$ largest continental US cities with noise volatility $\noisevol_{\edge} = 0.25$ for all $\edge\in\edges$ and affine cost functions of the form $\cost_{\edge}(\load_{\edge}) = a_{\edge}\load_{\edge} + b_{\edge}$ (both $a_{\edge}$ and $b_{\edge}$ drawn uniformly between $0$ and $1$).
The stochastic system \eqref{eq:SMD} was integrated numerically following a standard Euler\textendash Maruyama discretization scheme \cite{KE92} run for $N=10^{3}$ iterations with a step-size of $\delta=10^{-2}$.
Then, in \cref{fig:network-convergence}, we plotted the normalized total cost $\Cost_{0}(x) = \Cost(x) - \min\Cost$ in log-log scale:
in tune with \cref{thm:conv-var}, we see that if \eqref{eq:SMD} is run with a decreasing sensitivity parameter, the ergodic average $\bar X(t) = t^{-1}\int_{0}^{t} X(s) \dd s$ enjoys a power law convergence rate (corresponding to a straight line in log-log scale), even though the unrectified process $X(t)$ fails to converge altogether.

\subsection{Links with \acl{HR} gradient flows}
\label{sec:HR}

In this last section, we briefly examine some links between \eqref{eq:SMD} and the literature on \acl{HR} gradient flows \cite{BT03,ABB04,ABRT04}.
To begin with, when $h$ is steep and $\feas$ has non\-empty (topological) interior, the differential theory of Legendre transformations \cite[Chapter~26]{Roc70} shows that the mirror map $\mirror=\nabla h^{\ast}$ is a homeomorphism between $\dual = \dspace$ and $\intfeas = \dom\subd h$.
In this case, the system \eqref{eq:MD} induces a semiflow on $\intfeas$ via the dynamics
\begin{equation}
\label{eq:preHR}
\dot x
	= \frac{d}{dt} \mirror(y)
	= \nabla\mirror(y) \cdot \dot y
	= \nabla(\nabla h^{\ast}(y)) \cdot \payv(\mirror(y))
	= -\hess(h^{\ast}(y)) \cdot \nabla\obj(x).
\end{equation}
By Legendre's identity, we also have $\hess(h^{\ast}(y)) = \hess(h(\mirror(y)))^{-1}$ for all $y\in\dual$, so \eqref{eq:preHR} leads to the \acdef{HR} dynamics
\begin{equation}
\label{eq:HD}
\tag{HD}
\dot x
	= -H(x)^{-1} \cdot \nabla\obj(x),
\end{equation}
where $H(x) \equiv \hess(h(x))$ denotes the Hessian of $h$ evaluated at $x=\mirror(y)$.

As such, a natural question that arises is whether this equivalence between \eqref{eq:HD} and \eqref{eq:MD} carries over to the stochastic regime analyzed here.
To address this issue, assume first that the gradient input to \eqref{eq:HD} is perturbed by some random noise function $\epsilon(t)$ as in \eqref{eq:noise-Langevin},
viz.
\begin{equation}
\label{eq:HD-Langevin}
\dot x
	= H(x)^{-1}\cdot(-\nabla\obj(x) + \epsilon(t)).
\end{equation}
Then, writing out \eqref{eq:HD-Langevin} as a proper (Itô) \acl{SDE}, we get the stochastic \acl{HR} dynamics
\begin{equation}
\label{eq:SHD}
\tag{SHD}
dX
	= -H(X)^{-1}\cdot\nabla(\obj(X)) \dd t
	+ H(X)^{-1}\cdot d\noise,
\end{equation}
with $\noise(t)$ defined as in \eqref{eq:noise}.
On the other hand,
if $h$ is sufficiently smooth,
Itô's formula shows that the primal dynamics generated by \eqref{eq:SMD} on $\feas$ are given by
\begin{equation}
\label{eq:SMD-primal}
\tag{\ref*{eq:SMD}-P}
dX
	= \nabla(\mirror(Y)) \cdot \payv(X) \dd t
	+ \nabla(\mirror(Y)) \cdot d\noise
	+ \frac{1}{2} \covmat(X)\cdot\hess(\mirror(Y)) \dd t,
\end{equation}
with the last term corresponding to the second-order Itô correction induced by the nonlinearity of $\mirror$ (we have also taken $\temp=1$ for simplicity).

Comparing these two systems,
we see that the first two terms of \eqref{eq:SMD-primal} correspond precisely to the drift and diffusion coefficients of \eqref{eq:SHD}.
However, the Itô correction term $\frac{1}{2} \covmat(X)\cdot\hess(\mirror(Y)) \dd t$ (which involves the \emph{third} derivatives of $h^{\ast}$) has no equivalent in \eqref{eq:SHD}, meaning that \eqref{eq:SHD} and \eqref{eq:SMD-primal} \emph{do not coincide in general} \textendash\ that is, unless the mirror map $\mirror\from\dual\to\feas$ happens to be linear.

To illustrate this, take the linear objective $\obj(x) = x$ over $\feas=[0,1]$ and consider the dynamics generated by the entropic penalty function $h(x) = x\log x + (1-x) \log(1-x)$ with induced mirror map $\mirror(y) = e^{y}/(1+e^{y})$.
Then, \eqref{eq:SHD} becomes
\begin{equation}
\label{eq:SHD-x}
dX
	= -X(1-X) \left[ dt - \noisevol \dd W \right],
\end{equation}
while, after a routine application of Itô's lemma, \eqref{eq:SMD} gives
\begin{equation}
\label{eq:SMD-x}
dX
	= -X(1-X) \left[ dt - \noisevol\dd W \right]
	+ \frac{1}{2} X(1-X)(1-2X) \noisevol^{2} \dd t.
\end{equation}
We thus see that the primal dynamics \eqref{eq:SHD-x} and \eqref{eq:SMD-x} differ by the Itô correction term $\frac{1}{2}X(1-X)(1-2X) \dd t$.
Accordingly, the dynamics' behavior with respect to the minimizer $\sol=0$ of $\obj$ is expected to be different as well.

Indeed, the score process $Y(t)$ of \eqref{eq:SMD} becomes $Y(t) = Y(0) - t + \noisevol W(t) \to-\infty$ \as, implying in turn that $X(t)\to\sol$ under \eqref{eq:SMD-x}.
On the other hand, under \eqref{eq:SHD-x}, it can be shown that $X(t)$ converges to $\argmax\obj$ with high probability if $\noisevol$ is large enough.
To see this, let $G(x) = \log x - \log(1-x)$, so $G(X(t)) \to -\infty$ if $X(t)\to0^{+}$ and $G(X(t))\to+\infty$ if $X(t)\to 1^{-}$.
Itô's lemma then yields
\begin{equation}
\label{eq:dG}
dG
	= G'(X) \dd X + \frac{1}{2} (dX)^{2}
	= - \dd t + \noisevol\dd W + (X-1/2)\noisevol^{2} \dd t.
\end{equation}
From \eqref{eq:dG}, it is intuitively obvious (and can be shown rigorously) that the drift of \eqref{eq:dG} remains uniformly positive with probability arbitrarily close to $1$ if $X(0)>1/2$ and $\noisevol$ is large.%
\footnote{For a formal argument along these lines, see \cite[Theorem 3.3.3]{MV16}.}
In turn, this implies that $G(X(t))\to\infty$, i.e. \eqref{eq:SHD-x} converges with high probability to $\argmax\obj$ instead of $\argmin\obj$!

The above shows that the \acl{HR} system \eqref{eq:HD} is more vulnerable to noise compared to \eqref{eq:MD}.
Intuitively, this failure is due to the fact that \eqref{eq:HD} lacks an inherent ``averaging'' mechanism capable of dissipating the noise in the long run \textendash\ in \eqref{eq:MD}, this role is played by the direct aggregation of gradient steps up to time $t$.
Given the link between \acl{HR} dynamics and the replicator dynamics of evolutionary game theory \cite{HS98,ABB04}, this is also reminiscent of the different long-run behavior of the replicator dynamics with aggregate shocks \cite{Imh05} and the dynamics of stochastically perturbed exponential learning \cite{MM10}.
We intend to explore these relations at depth in a future paper.

\appendix

\section{Mirror maps and the Fenchel coupling}
\label{app:Fenchel}

In this appendix, we collect some basic properties of mirror maps and the Fenchel coupling.
We begin with a structural property of the inverse images of $\mirror$:

\begin{lemma}
\label{lem:inv-cone}
If $\mirror(y) = x$, then $\mirror(y+\payv) = x$ for all $\payv\in\pcone(x)$.
\end{lemma}

\begin{proof}
By \cref{prop:mirror}, it suffices to show that $y+\payv \in \pd h(x)$ for all $\payv\in\pcone(x)$.
However, since $\payv\in\pcone(x)$, we also have $\braket{\payv}{x' - x} \leq 0$ for all $x'\in\feas$, and hence
\begin{equation}
h(x')
	\geq h(x) + \temp\braket{y}{x' - x}
	\geq h(x) + \temp\braket{y + \payv}{x' - x},
\end{equation}
where the first inequality follows from the fact that $y \in \subd h(x)$.
The above shows that $y+\payv\in \pd h(x)$, so $\mirror(y+\payv) = x$, as claimed.
\end{proof}

The following technical comparison result is also useful in our analysis:

\begin{lemma}
\label{lem:Fench-polar}
If $y_{2} - y_{1} \in\pcone(\base)$, we have $\fench(\base,y_{1}) \geq \fench(\base,y_{2})$ and
\begin{equation}
\label{eq:Fench-polar}
\dnorm{y_{2} - y_{1}}
	\geq K \norm{\feas} \left[ \sqrt{1 + 2\delta/(K\norm{\feas}^{2})} - 1 \right],
\end{equation}
where $\delta = \fench(\base,y_{1}) - \fench(\base,y_{2})$.
\end{lemma}

\begin{proof}
Let $\payv = y_{2} - y_{1}$ and set $g(t) = \fench(\base,y_{1}+t\payv)$, $t\in[0,1]$.
Differentiating yields $g'(t) = \braket{\payv}{\mirror(y_{1}+t\payv) - \base} \leq 0$ for all $t$ because $\payv\in\pcone(\base)$ and $\mirror(y_{1}+t\payv) - \base\in\tcone(\base)$.
We thus get $\fench(\base,y_{2}) = \fench(\base,y_{1}+\payv) \leq \fench(\base,y_{1})$, as claimed.

For our second assertion, \eqref{eq:Fench-diff} readily yields
\begin{flalign}
\label{eq:Fench-quad-bound}
\fench(\base,y_{2}) - \fench(\base,y_{1})
	&\leq \braket{y_{2} - y_{1}}{\mirror(y) - \base} + \frac{1}{2K} \dnorm{y_{2} - y_{1}}^{2}
	\notag\\
	&\leq \dfeas\:\dnorm{y_{2} - y_{1}} + \frac{1}{2K} \dnorm{y_{2} - y_{1}}^{2},
\end{flalign}
and, after rearranging, we get $\omega^{2} + 2K\dfeas \omega - 2K\delta \geq 0$, where $\omega = \dnorm{y_{2} - y_{1}} \geq 0$.
The roots of this inequality are $\omega_{\pm} = -K\dfeas \pm \sqrt{K^{2} \dfeas^{2} + 2K\delta}$, so $\omega_{-} < 0 \leq \omega_{+}$.
This implies that \eqref{eq:Fench-quad-bound} only holds if $\omega \geq \omega_{+}$, so \eqref{eq:Fench-polar} follows.
\end{proof}

Our next result describes the evolution of the $\temp$-deflated Fenchel coupling $\lyap(t) = \temp^{-1} \fench(\sol,\temp y(t))$ under \eqref{eq:MD}:

\begin{lemma}
\label{lem:dFench-det}
Fix some $\sol\in\feas$.
Then, under \eqref{eq:MD}, we have
\begin{equation}
\label{eq:dFench-det}
\dot\lyap(t)
	= \braket{\payv(x(t))}{x(t) - \sol}.
\end{equation}
Consequently, $\lyap(t)$ is nonincreasing for all $\sol\in\argmin\obj$.
\end{lemma}

\begin{proof}
By the definition \eqref{eq:Fench-temp} of the $\temp$-deflated Fenchel coupling and \cref{prop:mirror}, we have
\begin{equation}
\dot\lyap(t)
	= \temp^{-1} \left[ \braket{\temp\dot y}{\nabla h^{\ast}(\temp y)} \right] - \braket{\dot y}{\sol}
	= \braket{\payv(x)}{x - \sol},
\end{equation}
as claimed.
As for our second claim, simply note that $\braket{\payv(x)}{x - \sol} \leq f(\sol) - f(x) \leq 0$ for all $\sol\in\argmin\obj$.
\end{proof}

We now extend \cref{lem:dFench-det} to the stochastic dynamics \eqref{eq:SMD} with a variable sensitivity parameter $\temp\equiv\temp(t)$:

\begin{lemma}
\label{lem:dFench-stoch}
Fix some $\sol\in\feas$.
Then, for all $t \geq t_{0} \geq 0$, we have
\begin{subequations}
\label{eq:dFench-stoch}
\begin{flalign}
\lyap(t) - \lyap(t_{0})
	&\label{eq:dFench-drift}
	\leq \int_{t_{0}}^{t} \braket{\payv(X(s))}{X(s) - \sol} \dd s
	\vphantom{\sum_{i=1}^{n}}
	\\
	&\label{eq:dFench-temp}
	- \int_{t_{0}}^{t} \frac{\dot\temp(s)}{\temp(s)^{2}} [h(\sol) - h(X(s))] \dd s
	\\
	&\label{eq:dFench-Ito}
	+ \frac{1}{2K} \int_{t_{0}}^{t} \temp(s) \trof{\covmat(X(s),s)} \dd s
	\\
	&\label{eq:dFench-noise}
	+ \sum_{i=1}^{n} \int_{t_{0}}^{t} (X_{i}(s) - \sol_{i}) \dd Z_{i}(s).
\end{flalign}
\end{subequations}
\end{lemma}

\begin{proof}
By \cref{prop:mirror}, we have $\nabla\fench(\sol,y) = \nabla h^{\ast}(y) - \sol = \mirror(y) - \sol$ for all $y\in\dual$.
Thus, given that $\mirror = \nabla h^{\ast}$ is $(1/K)$-Lipschitz continuous (again by \cref{prop:mirror}), our result follows from \cref{prop:Ito} (see also \cref{rem:Ito}).
\end{proof}

\section{Convergence analysis}
\label{app:proofs}

In this appendix, we prove the convergence results of \cref{sec:prelims,sec:results}.

\subsection{Deterministic analysis}
\label{app:det}

We begin with the convergence properties of the deterministic dynamics \eqref{eq:MD}:

\begin{proofof}{Proof of \cref{thm:conv-det}}
For all $\sol\in\argmin\obj$, \cref{lem:dFench-det} gives
\begin{flalign}
\label{eq:Fbound-det1}
\lyap(t) - \lyap(0)
	&= \int_{0}^{t} \braket{\payv(x(s))}{x(s) - \sol} \dd s
	\leq t [ \min\obj - \bar\obj(t) ].
\end{flalign}
A simple rearrangement yields $\bar\obj(t) - \min\obj \leq \lyap(0)/t$, so the bound for $\obj_{\min}(t)$ follows trivially.
As for the specific rate $\depth/t$, it suffices to note that $\fench(\sol,0) = h(\sol) + h^{\ast}(0) = h(\sol) - h(\mirror(0)) \leq \max\setdef{h(x') - h(x)}{x,x'\in\feas}$.

For our second assertion, let $\olimit$ be an $\omega$-limit of $x(t)$ and assume that $\olimit\notin\argmin\obj$.
Since $\argmin\obj$ is closed, there exists a neighborhood $U$ of $\olimit$ in $\feas$ such that $\braket{\payv(x)}{x - \sol} \leq -a$ for some $a>0$ and for all $\sol\in\argmin\obj$.
Furthermore, since $\olimit$ is an $\omega$-limit of $x(t)$, there exists an increasing sequence of times $t_{k}\uparrow\infty$ such that $x(t_{k})\in U$ for all $k$.
Then, for all $\tau>0$, \cref{prop:mirror} gives
\begin{flalign}
\label{eq:xdiff}
\norm{x(t_{k}+\tau) - x(t_{k})}
	&= \norm{\mirror(\temp y(t_{k}+\tau)) - \mirror(\temp y(t_{k}))}
	\leq \frac{\temp}{K} \dnorm{y(t_{k}+\tau) - y(t_{k})}
	\notag\\
	&\leq \frac{\temp}{K} \int_{t_{k}}^{t_{k}+\tau} \dnorm{\payv(x(s))} \dd s
	\leq \frac{\temp\tau}{K} \max_{x\in\feas} \dnorm{\payv(x)}.
\end{flalign}

Given that the bound \eqref{eq:xdiff} does not depend on $k$, there exists some sufficiently small $\delta>0$ such that $x(t_{k} + \tau) \in U$ for all $\tau\in[0,\delta]$, $k\in\N$ (so we also have $\braket{\payv(x(t_{k} + \tau))}{x(t_{k}+\tau) - \sol} \leq -a$).
Therefore, given that $\braket{\payv(x)}{x - \sol} \leq 0$ for all $x\in\feas$, $\sol\in\argmin\obj$, we get
\begin{flalign}
\label{eq:Fbound-det2}
\lyap(t_{k}+\delta) - \lyap(0)
	&\leq \sum_{j=1}^{k} \int_{t_{j}}^{t_{j} + \delta} \braket{\payv(x(s))}{x(s) - \sol} \dd s
	\leq - ak\delta,
\end{flalign}
showing that $\liminf_{t\to\infty} \fench(\sol,\temp y(t)) = -\infty$, a contradiction.
Since $x(t)$ admits at least one $\omega$-limit, we conclude that $x(t)$ converges to $\argmin\obj$.

Assuming $\sol\in\argmin\obj$ is an $\omega$-limit of $x(t)$, we have $x(t_{k}')\to\sol$ for some sequence of times $t_{k}'\uparrow\infty$.
By \eqref{eq:Fench-reg}, it follows that $\lyap(t_{k}')\to0$ and hence, with $\lyap(t)$ nonincreasing, that $\lyap(t)\to0$.
Since $x(t)$ admits at least one $\omega$-limit (by the compactness of $\feas$), we conclude that $\lim_{t\to\infty} x(t) = \sol$, as claimed.
\end{proofof}

\subsection{The vanishing noise limit}
\label{app:smallnoise}

We proceed with the proof of our ``vanishing noise'' results, namely \cref{prop:smallnoise,thm:smallnoise}:

\begin{proofof}{Proof of \cref{prop:smallnoise}}
Arguing by contradiction, assume that $X(t)$ remains a bounded distance away from $\argmin\obj$ for large $t$ with positive probability.
This implies that there exists some $a>0$ and a (random) $t_{0}$ such that
\begin{equation}
\braket{\payv(X(t))}{X(t) - \sol} \leq - a
	\quad
	\text{for all $t\geq t_{0}$},
\end{equation}
again with positive probability.
Then, fixing some $\sol\in\argmin\obj$ and taking the associated Fenchel coupling $\lyap(t) = \temp^{-1} \fench(\sol,\temp Y(t))$, \cref{lem:dFench-stoch} gives
\begin{flalign}
\label{eq:Fench-bound}
\lyap(t) - \lyap(t_{0})
	&\leq \int_{t_{0}}^{t} \braket{\payv(X(s))}{X(s) - \sol} \dd s
	\notag\\
	&+ \frac{1}{2K} \int_{t_{0}}^{t} \trof{\covmat(X(s),s)} \dd s
	+ \sum_{i=1}^{n} \int_{t_{0}}^{t} (X_{i}(s) - \sol) \dd Z_{i}(s)
	\notag\\
	&\leq - a(t - t_{0})
	+  \frac{1}{2K} \int_{t_{0}}^{t} \trof{\covmat(X(s),s)} \dd s
	+ \snoise(t),
\end{flalign}
where
$\snoise(t)$ denotes the martingale term $\sum_{i=1}^{n} \int_{t_{0}}^{t} (X_{i}(s) - \sol_{i}) \dd Z_{i}(s)$.
Since $\norm{X(s) - \sol} \leq \norm{\feas} < \infty$, \cref{lem:Wbound} in \cref{app:stoch} shows that $\snoise(t)/t \to 0$ \as.
Moreover, we also have
$\lim_{t\to\infty} t^{-1} \int_{0}^{t} \trof{\covmat(X(s),s)} \dd s = \lim_{t\to\infty} \fnorm{\noisevol(X(t),t)}^{2} = 0$
(by the vanishing noise assumption and de l'Hôpital's rule), so the last two terms in \eqref{eq:Fench-bound} are both sublinear in $t$.
We thus obtain $\lyap(t) \to -\infty$ with positive probability, a contradiction which establishes our claim.
\end{proofof}

We are now in a position to prove \cref{thm:smallnoise} under the additional assumption $\sup_{x\in\feas} \fnorm{\noisevol(x,t)} = o(1/\sqrt{\log t})$:

\begin{proofof}{Proof of \cref{thm:smallnoise}}
Without loss of generality, assume that $\temp=1$;
otherwise, simply replace $h$ by $\temp^{-1}h$ in the definition of \eqref{eq:SMD}.
Also, for simplicity, we only prove the case where $\obj$ admits a unique minimizer $\sol\in\feas$;
the general argument is similar (but more cumbersome to write down), so we omit it.

To begin, fix some $\eps>0$ and let $U_{\eps} = \setdef{x=\mirror(y)}{\fench(\sol,y) < \eps}$.
Our first claim is that there exists a time $T\equiv T(\eps)$ such that $\fench(\sol,\semiflow_{T}(y)) \leq \max\{\eps,\fench(\sol,y) - \eps\}$ for all $y\in\dual$.
Indeed, by \eqref{eq:Fench-reg} and the continuity of $\payv(x)$, there exists some $a \equiv a(\eps) > 0$ such that
\begin{equation}
\braket{\payv(x)}{x - \sol}
	\leq -a
	\quad
	\text{for all $x\notin U_{\eps}$}.
\end{equation}
Consequently, if $\tau_{y} = \inf\setdef{t>0}{\mirror(\semiflow_{t}(y)) \in U_{\eps}}$ is the first time at which an orbit of \eqref{eq:MD} hits $U_{\eps}$, \cref{lem:dFench-det} gives
\begin{equation}
\label{eq:Fench-bound1}
\fench(\sol,\semiflow_{t}(y)) - \fench(\sol,y)
	= \int_{0}^{t} \braket{\payv(x(s))}{x(s) - \sol} \dd s
	\leq -at
	\quad
	\text{for all $t\leq\tau_{y}$}.
\end{equation}
In view of this, set $T=\eps/a$ and consider the following cases:
\begin{enumerate}
\item
If $T\leq\tau_{y}$, \eqref{eq:Fench-bound1} gives $\fench(\sol,\semiflow_{T}(y)) \leq \fench(\sol,y) - \eps$.
\item
If $T>\tau_{y}$, we have $\fench(\sol,\semiflow_{T}(y)) \leq \fench(\sol,\semiflow_{\tau_{y}}(y)) = \eps$ (recall here that $\fench(\sol,\semiflow_{t}(y))$ is weakly decreasing in $t$).
\end{enumerate}
In both cases we have $\fench(\sol,\semiflow_{T}(y)) \leq \max\{\eps,\fench(\sol,y) - \eps\}$, as claimed.

Now, let $(Y(t))_{t\geq0}$ be a solution of \eqref{eq:SMD};
we then claim that $Y(t)$ is \as an \acl{APT} of $\eqref{eq:MD}$ in the sense of \cref{def:APT}.
Indeed, by Proposition 4.6 in \cite{Ben99}, it suffices to show that $\int_{0}^{\infty} e^{-c/\covmax(t)} \dd t < \infty$
where $\covmax(t) = \sup_{x\in\feas} \trof{\covmat(x,t)}$ and $c>0$ is arbitrary.
However, by assumption
\begin{equation}
\covmat_{\max}(t)
	= \sup_{x\in\feas} \fnorm{\noisevol(x,t)}^{2}
	= \phi(t)/\log t
\end{equation}
for some $\phi(t)$ with $\lim_{t\to\infty}\phi(t) = 0$.
Therefore,
\begin{equation}
e^{-c/\covmat_{\max}(t)}
	= \parens*{e^{\log t}}^{-c/\phi(t)}
	= t^{-c/\phi(t)}
	= \bigoh(t^{-\beta})
	\quad
	\text{for all $\beta>1$},
\end{equation}
and our assertion follows.


To proceed, fix a solution $Y(t)$ of \eqref{eq:SMD} which is an \ac{APT} of \eqref{eq:MD}.
Moreover, with notation as in \cref{def:APT},  let $\delta \equiv \delta(\eps)$ be such that $\delta\norm{\feas} + \delta^{2}/(2K) \leq \eps$
and
choose some (random) $t_{0}\equiv t_{0}(\eps)$ such that $\sup_{0\leq h\leq T} \dnorm{Y(t+h) - \semiflow_{h}(Y(t))} \leq \delta$ for all $t\geq t_{0}$.
Then, for all $t\geq t_{0}$, we get
\begin{flalign}
\label{eq:Fench-3eps}
\fench(\sol,Y(t+h))
	&\leq \fench(\sol,\semiflow_{h}(Y(t)))
	+ \braket{Y(t+h) - \semiflow_{h}(Y(t))}{\mirror(\semiflow_{h}(Y(t))) - \sol}
	\notag\\
	&+ \frac{1}{2K} \dnorm{Y(t+h) - \semiflow_{h}(Y(t))}^{2}
	\notag\\
	&\leq \fench(\sol,\semiflow_{h}(Y(t)))
	+ \delta \norm{\feas}
	+ \frac{\delta^{2}}{2K}
	\leq \fench(\sol,\semiflow_{h}(Y(t))) + \eps,
\end{flalign}
where, in the first line, we used the second-order Taylor estimate for the Fenchel coupling derived in \cref{prop:Fenchel} (cf.~\cref{app:Fenchel}).

By \cref{prop:smallnoise}, there exists some $T_{0}\geq t_{0}$ such that $\fench(\sol,Y(T_{0})) \leq 2\eps$ \as, implying that $\fench(\sol,Y(T_{0})) \leq 2\eps$ for some $T_{0}\geq t_{0}$.
Hence, by \eqref{eq:Fench-3eps}, we get
\begin{equation}
\fench(\sol,Y(T_{0}+h))
	\leq \fench(\sol,\semiflow_{h}(Y(T_{0}))) + \eps
	\leq \fench(\sol,Y(T_{0})) + \eps
	\leq 3\eps
\end{equation}
for all $h\in[0,T]$.
However,
we also have $\fench(\sol,\semiflow_{T}(Y(T_{0}))) \leq \max\{\eps,\fench(\sol,Y(T_{0})) - \eps\} \leq \eps$, so $\fench(\sol,Y(T_{0}+T)) \leq \fench(\sol,\semiflow_{T}(Y(T_{0}))) + \eps \leq 2\eps$.
Therefore, repeating the above argument at $T_{0}+T$ (instead of $T_{0}$) and proceeding inductively, we get $\fench(\sol,Y(T_{0}+h)) \leq 3\eps$ for all $h\in[kT,(k+1)T]$, $k\in\N$.
With $\eps$ arbitrary, we conclude that $\fench(\sol,Y(t)) \to 0$, so $X(t)\to\sol$, as claimed.
\end{proofof}

\subsection{Long-run concentration around solution points}
\label{app:invariant}

We now turn to the ergodic properties of \eqref{eq:SMD} under persistent, nonvanishing noise:

\begin{proofof}{Proof of \cref{prop:hitting}}
Let $\lyap(t) \equiv \temp^{-1} \fench(\sol,\temp Y(t))$ denote the $\temp$-deflated Fenchel coupling between $\sol$ and $Y(t)$.
Then, by the growth bound \eqref{eq:dFench-stoch}, we get
\begin{flalign}
\label{eq:dFench-strong1}
\lyap(t) - \lyap(0)
	&\leq \int_{0}^{t} \braket{\payv(X(s))}{X(s) - \sol} \dd s
	+ \frac{1}{2K} \int_{0}^{t} \temp \trof{\covmat(X(s),s)} \dd s
	+ \snoise(t)
	\notag\\
	&\leq -\frac{\strong}{2} \int_{0}^{t} \norm{X(s) - \sol}^{2} \dd s
	+ \frac{\temp \noisevar t}{2K}
	+ \snoise(t),
\end{flalign}
where $\snoise(t) = \sum_{i=1}^{n} \int_{0}^{t} (X_{i}(s) - \sol_{i}) \dd Z_{i}(s)$
and we used the strong convexity bound \eqref{eq:strong} to write $\braket{\payv(x)}{x - \sol} \leq \obj(\sol) - \obj(x) \leq -\frac{1}{2} \strong\norm{x-\sol}^{2}$ in the second line.
Since $\lyap(t)\geq0$, the bound \eqref{eq:MSE-mean} follows by taking expectations,
exploiting the fact that $\snoise(t)$ has zero mean,
and rearranging.

Now, replacing $t$ by $\tau_{\delta}\wedge t$ in \eqref{eq:dFench-strong1}, we also get
\begin{flalign}
\exof{\lyap(\tau_{\delta}\wedge t)}
	&\leq \lyap(0)
	- \frac{\strong}{2} \exof*{\int_{0}^{\tau_{\delta}\wedge t} \norm{X(s) - \sol}^{2} \dd s}
	+ \frac{\temp\noisevar}{2K} \exof{\tau_{\delta}\wedge t}
	\notag\\
\label{eq:dFench-strong3}
	&\leq \lyap(0)
	+ \frac{\temp\noisevar - \strong K\delta^{2}}{2K} \exof{\tau_{\delta}\wedge t},
\end{flalign}
where we used the fact that $\norm{X(s) - \sol} \geq \delta$ for all $s\leq\tau_{\delta}$.
Since $\lyap \geq 0$, we conclude that $\exof{\tau_{\delta}\wedge t} \leq 2K\lyap(0)/(\strong K\delta^{2} - \temp\noisevar)$.
Our claim then follows by letting $t\to\infty$ (so $\tau_{\delta}\wedge t\to\tau_{\delta}$) and invoking the dominated convergence theorem.
Finally, the optimized bound \eqref{eq:hitting-bound-opt} is obtained by maximizing the denominator of \eqref{eq:hitting-bound}.
\end{proofof}

We are now in a position to estimate the occupation measure of $X(t)$:

\begin{proofof}{Proof of \cref{thm:invariant}}
We begin by introducing a transformed version of $Y(t)$ which is recurrent under \eqref{eq:SMD}.%
\footnote{Recall here that $Y(t)$ is \emph{recurrent} if there exists a compact set $\cpt$ such that $\probof{Y(t)\in\cpt \; \text{for some} \; t\geq0} = 1$ for every initial condition $y_{0}$ of $Y$ \cite{Kha12,Bha78}.
In our case, the set $\mirror^{-1}(\ball_{\delta})$ need not be compact, so the generating process $Y(t)$ need not be recurrent either.}
To that end, note first that $\mirror^{-1}(x)$ always contains a translate of the polar cone $\pcone(x)$ of $\feas$ at $x$ (cf.~\cref{lem:inv-cone});
in particular, if $\feas$ is not full-dimensional, $\mirror^{-1}(\ball_{\delta})$ contains a nonzero affine subspace of $\dual$.
To mod out this subspace, let $\subspace = \aff(\feas-\feas) \subseteq \vecspace$ denote the smallest subspace of $\vecspace$ that contains $\feas$ when translated to the origin (so $\feas$ may be considered as a convex body of $\subspace$).
Then, writing $\subdual\equiv\dsubspace$ for the dual space of $\subspace$, define the restriction map $\dincl\from\dual\to\subdual$ as
\begin{equation}
\label{eq:restriction}
\braket{\dincl(y)}{z}
	= \braket{y}{z}
	\quad
	\text{for all $z\in\subspace$}.
\end{equation}
We then have $\dincl(y) = 0$ whenever $y$ annihilates $\subspace$ (i.e. $\braket{y}{z} = 0$ for all $z\in\subspace$).%
\footnote{Of course, if $\feas$ has nonempty interior as a subset of $\vecspace$, we have $\subspace = \vecspace$ and $\dincl$ is the identity.}

Accordingly, in view of \cref{prop:hitting}, it stands to reason that the transformed process $\Psi(t) = \dincl(Y(t))$ is recurrent.
Indeed, from \cite[Proposition 3.1]{Bha78}, it suffices to show that
\begin{inparaenum}
[\itshape a\upshape)]
\item
$\Psi(t)$ is an Itô diffusion whose infinitesimal generator is uniformly elliptic;
and
\item
there exists some compact set $\subcpt\subseteq\subdual$ such that $\probof{\text{$\Psi(t)\in\subcpt$ for some $t\geq0$}} = 1$ for every initial condition $\psi_{0}\in\subdual$.
\end{inparaenum}
The rest of our proof is devoted to establishing these two requirements.

For the first, write $\dincl(y)$ in coordinates as $(\dincl(y))_{i} = \sum_{k=1}^{n} \Pi_{ik} y_{k}$.
Then, with $\Psi = \Pi\cdot Y$, we get
\begin{equation}
\label{eq:SMD-mod}
d\Psi_{i}
	= \sum_{k=1}^{n} \Pi_{ik} \, \parens*{ \payv_{k}(X) \dd t + \dd Z_{k} }.
\end{equation}
Moreover, define the ``restricted'' mirror map $\submirror\from\subdual\to\feas$ as
\begin{equation}
\label{eq:mirror-mod}
\submirror(w)
	= \argmax\nolimits_{x\in\feas}\{ \braket{w}{x} - h(x) \},
\end{equation}
where, in a slight abuse of notation, $\feas$ is treated as a subset of $\subspace$.
By definition, we have $\braket{y}{x} = \braket{\dincl(y)}{x}$ for all $x\in\feas$, so $\argmax\{\braket{y}{x} - h(x)\} = \argmax\{\braket{\dincl(y)}{x} - h(x)\}$ for all $y\in\dual$.
This shows that $X(t)$ can be expressed as $X(t) = \submirror(\temp \dincl(Y(t))) = \submirror(\temp \Psi(t))$, so \eqref{eq:SMD-mod} represents a regular Itô diffusion.

We now claim that the infinitesimal generator $\gen_{\Psi}$ of $\Psi$ is uniformly elliptic.
Indeed, the quadratic covariation of $\Psi$ is given by
\begin{equation}
d[\Psi_{i},\Psi_{j}]
	= \dd(\Pi Y)_{i} \dd(\Pi Y)_{j}
	= \sum_{k,\ell = 1}^{n} \Pi_{ik} \Pi_{j\ell} \covmat_{k\ell} \dd t
	= \parens*{\Pi \covmat \Pi^{\top}}_{ij} \dd t,
\end{equation}
where we used the definition \eqref{eq:variation} of $\covmat$ in the penultimate equality.
However, we also have
\(
\Pi\covmat\Pi^{\top}
	\mgeq \lambda \Pi\Pi^{\top}
	\mgeq \lambda \pi_{\min}^{2} I,
\)
where $\pi_{\min} > 0$ denotes the smallest singular value of $\Pi^{\top}$ (recall that $\dincl$ has full rank).
This shows that the principal symbol $\Pi\covmat\Pi^{\top}$ of $\gen_{\Psi}$ is uniformly positive-definite, so $\gen_{\Psi}$ is uniformly elliptic.

For the second component of our proof, assume without loss of generality that $\delta$ is sufficiently small so $\ball_{\delta}\subseteq\intfeas$
(obviously, this is possibly only if $\sol\in\intfeas$).
Momentarily viewing $\feas$ as a convex body of $\subspace$ (and $\ball_{\delta}$ as a ball in $\subspace$), Remark 6.2.3 in \cite{HUL01} implies that the set $\subcpt = \temp^{-1}\subd h(\ball_{\delta})$ is compact.%
\footnote{Strictly speaking, Remark 6.2.3 of \cite{HUL01} applies to convex functions that are defined on all of $\subspace$, but since this is a local property, it is trivial to extend it to our case.}
Then, by \cref{prop:hitting}, it follows that $\Psi(t)$ hits $\subcpt$ in finite time \as for every initial condition $\psi_{0} = \dincl(y_{0}) \in\subdual$.

Since the generator $\gen_{\Psi}$ of $\Psi$ is uniformly elliptic and $\subcpt$ is compact, Proposition~3.1 in \cite{Bha78} shows that $\Psi(t)$ is recurrent.
Hence, from standard results in the theory of Itô diffusions \cite[Theorem~4.4.1, Theorem~4.4.2 and Corollary~4.4.4]{Kha12}, we conclude that $\Psi(t)$ admits a unique invariant distribution $\nu$ which satisfies the law of large numbers
\begin{equation}
\lim_{t\to\infty} t^{-1} \int_{0}^{t} \phi(\Psi(s)) \dd s
	=\int_{\subdual} \phi \dd\nu,
\end{equation}
for every $\nu$-integrable function $\phi$ on $\subdual$.
We thus obtain
\begin{flalign}
\lim_{t\to\infty} \frac{1}{t} \int_{0}^{t} \one(X(s) \in \ball_{\delta}) \dd s
	&= \lim_{t\to\infty} \frac{1}{t} \int_{0}^{t} \one(\temp\Psi(s) \in \submirror^{-1}(\ball_{\delta})) \dd s
	\notag\\
	&= \int_{\subdual} \one_{\temp^{-1}\submirror^{-1}(\ball_{\delta})} \dd\nu
	= \nu(\temp^{-1} \submirror^{-1}(\ball_{\delta})),
\end{flalign}
i.e. $\occ_{t}(\ball_{\delta}) \to \nu(\temp^{-1}\submirror^{-1}(\ball_{\delta}))$ as $t\to\infty$ \as.
Similarly, given that the limit of $\occ_{t}$ is deterministic and finite, the mean square bound \eqref{eq:MSE-mean} also yields
\begin{flalign*}
1 - \nu(\temp^{-1}\submirror^{-1}(\ball_{\delta}))
	&= \lim_{t\to\infty} \frac{1}{t} \exof*{\int_{0}^{t} \one(X(s) \notin \ball_{\delta}) \dd s}
	\notag\\
	&\leq \lim_{t\to\infty} \frac{1}{t} \exof*{\int_{0}^{t} \frac{\norm{X(s) - \sol}^{2}}{\delta^{2}} \dd s}
	\notag\\
	&\leq \lim_{t\to\infty} \frac{1}{\delta^{2}} \bracks*{\frac{2\fench(\sol,\temp y_{0})}{\temp\strong t} + \frac{\temp\noisevar}{\strong K}}
	= \frac{\temp\noisevar}{\strong K\delta^{2}},
\end{flalign*}%
as was to be shown.
\end{proofof}

\subsection{Convergence to sharp solutions}
\label{app:sharp}

The proof of our convergence result for sharp solutions is fairly involved, so we encode it in a series of technical lemmas.
The first one shows that neighborhoods of sharp solutions are recurrent under \eqref{eq:SMD}:

\begin{lemma}
\label{lem:sharp-recurrence}
Fix $\delta>0$ and assume that $\obj$ admits a $\sharp$-sharp solution.
If \eqref{eq:SMD} is run with sensitivity parameter $\temp < 2\sharp\delta K/\noisevar$, there exists a \textpar{random} sequence of times $t_{n}\uparrow\infty$ such that $\norm{X(t_{n}) - \sol} < \delta$ for all $n$ \as.
\end{lemma}

\begin{proof}
Suppose there exists some (random) $t_{0}$ such that $\norm{X(t) - \sol} \geq \delta$ for all $t\geq t_{0}$.
Then, writing $\lyap(t) = \temp^{-1} \fench(\sol,\temp Y(t))$ for the $\temp$-deflated Fenchel coupling between $\sol$ and $Y(t)$, \cref{lem:dFench-stoch} yields
\begin{flalign}
\label{eq:dFench-sharp}
\lyap(t)
	&\leq \lyap(t_{0})
	+ \int_{t_{0}}^{t} \braket{\payv(X(s))}{X(s) - \sol} \dd s
	+ \frac{1}{2K} \int_{t_{0}}^{t} \temp \trof{\covmat(X(s),s)} \dd s
	+ \snoise(t)
	\notag\\
	& \leq \lyap(t_{0})
	- \bracks*{\sharp\delta - \frac{\temp\noisevar}{2K} - \frac{\snoise(t)}{t - t_{0}}} (t - t_{0}),
\end{flalign}
where we set $\snoise(t) = \sum_{i=1}^{n} \int_{t_{0}}^{t} (X_{i}(s) - \sol_{i}) \dd Z_{i}(s)$ in the first line
and we used \cref{lem:sharp} in the second.
Since $\snoise(t)/(t - t_{0}) \to 0$ by \cref{lem:Wbound} in \cref{app:stoch}, the bound \eqref{eq:dFench-sharp} yields $\lim_{t\to\infty} \lyap(t) = -\infty$ if $\temp\noisevar < 2\sharp\delta K$, a contradiction (recall that $\lyap(t)\geq0$ for all $t\geq0$).
This shows that $t_{0} = \infty$ \as, so there exists a sequence $t_{n}\uparrow\infty$ such that $\norm{X(t_{n}) - \sol} < \delta$ for all $n$.
\end{proof}

Our next result shows that the dual process $Y(t)$ keeps moving roughly along the direction of $\payv(\sol)$ with probability arbitrarily close to $1$ if $\temp$ is chosen small enough and $X(0)$ starts sufficiently close to $\sol$.

\begin{lemma}
\label{lem:escape}
Suppose that $\obj$ admits a sharp minimum $\sol\in\feas$, and let $\cone$ be a polyhedral cone such that $\payv(\sol)\in\intr(\cone)$ and $\cone\subseteq\intr(\pcone(\sol)) \cup\{0\}$.
Then,
for small enough $\temp,\eps,\delta>0$,
and for every initial condition $y_{0}\in\dual$ with $\fench(\sol, \temp y_{0})<\eps$,
there exists some $\ybase\in\dual$
such that
$\fench(\sol, \temp \ybase) = \eps + \delta$
and
\begin{equation}
\label{eq:escape-prob}
\probof{\text{$Y(t) \in \ybase + \cone$ for all $t\geq0$}}
	\geq 1 - e^{-\kappa\delta/(\temp\noisevar)},
\end{equation}
where $\kappa>0$ is a constant that depends only on $\cone$ and $\obj$.
\end{lemma}


\begin{figure}
\centering
\footnotesize

\colorlet{TangentColor}{DodgerBlue!40!MidnightBlue}
\colorlet{PolarColor}{FireBrick}

\begin{tikzpicture}
[>=stealth,
vecstyle/.style = {->, line width=.5pt},
edgestyle/.style={-, line width=.5pt},
nodestyle/.style={circle, fill=Black,inner sep = .75pt},
plotstyle/.style={color=DarkGreen!80!Cyan,thick}]

\def\radius{6}
\def\costhirty{0.8660256}
\def\cosfortyfive{0.7071068}
\def\cosseventyfive{0.2588190451}
\def\diff{0.77886966103678}
\def\veclength{.5}
\def\upangle{130}
\def\downangle{255}
\def\diffangle{60}
\def\midangle{.55*\upangle+.45*\downangle-180}
\def\conescale{.8}

\coordinate (base) at (-2,-1/3);
\coordinate (start) at (0.4,.2);
\coordinate (y1) at (0.592494, -0.019825);
\coordinate (y2) at (0.792766, 0.0429555);
\coordinate (y3) at (0.948162, .081878);
\coordinate (y4) at (1.39485, -0.090005);
\coordinate (y5) at (2.15562, 0.08761);

\draw [edgestyle] (base.center) -- ++(\upangle-90:\conescale*\radius);
\draw [edgestyle] (base.center) -- ++(\downangle+90:\conescale*\radius) node [right] {$\pcone(\sol)$};

\fill [PolarColor!5] (base.center) ++(\downangle+105:\conescale*\radius) -- (base.center) --++(\upangle-105:\conescale*\radius);
\draw [edgestyle, PolarColor] (base.center) -- ++(\upangle-105:\conescale*\radius) node [right] {$\cone$};
\draw [edgestyle, PolarColor] (base.center) -- ++(\downangle+105:\conescale*\radius);

\draw [vecstyle] (base.center) -- (\midangle:-.5) node [above, black] {$\quad\payv(\sol)$};

\draw
[densely dashed]
($(base) + (-.35,1.75)$)
node [above] {$\fench(\sol,\temp y) = \eps + \delta$}
.. controls ($(base) + (-.3,1)$)
and ($(base) + (-.25,.5)$)
.. (base)
.. controls ($(base) + (.25,-.5)$)
and ($(base) + (.75,-.75)$)
.. ($(base) + (1.25,-1)$);

\draw
[densely dashed]
($(start) + (-.5,2)$) node [above] {$\fench(\sol,\temp y) = \eps$} .. controls ($(start) + (-.5,0)$) and ($(start) + (-.5,0)$) .. ($(start) + (1,-1.25)$);

\node [nodestyle, label = left:$\ybase$] (base) at (base.north) {};
\node [nodestyle, label = above:$y_{0}$] (start) at (start) {};

\draw [decorate, decoration ={random steps, segment length=1.5,amplitude=1pt}]
(start.center) -- (y1) -- (y2) -- (y3) -- (y4) -- (y5) node [above] {$Y(t)$};

\end{tikzpicture}
\caption{The various sets in the proof of \cref{lem:escape}.}
\label{fig:cones}
\end{figure}
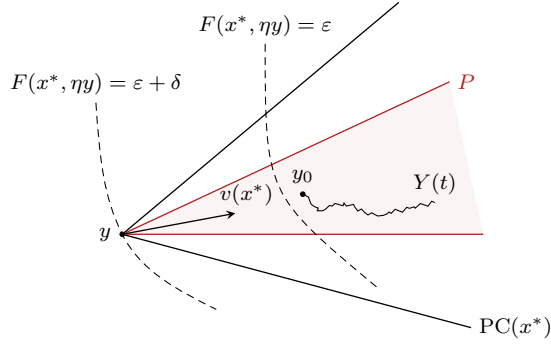


\begin{proof}
Let $\cone^{\perp} = \setdef{z\in\vecspace}{\braket{y}{z}\leq 0\; \text{for all}\; y\in\cone}$ denote the polar cone of $\cone$ and let $\genvecs = \braces{\genvec_{j}}_{j=1}^{d}$ be a basis for $\cone^{\perp}$ (recall that $\cone$ is assumed polyhedral).
Further,
fix a small compact neighborhood $\nhd$ of $\sol$ such that $\braket{\payv(x)}{z} \leq - \sharp_{\nhd} \norm{z}$ for some $\sharp_{\nhd}>0$ and all $x\in \nhd$, $z\in\cone^{\perp}$;%
\footnote{That such a $\sharp_{\nhd}$ exists is a consequence of the continuity of $\payv(x)$ and \cref{lem:sharp}.}
with a fair bit of hindsight, assume also that $\delta < K\norm{\feas}^{2}$ is sufficiently small so that $\mirror(\temp y)\in\nhd$ whenever $\fench(\sol,\temp y) \leq \eps + \delta$.
Finally, invoking \cref{lem:Fench-polar}, let $\ybase = y_{0} - c\payv(\sol)$ for some $c>0$ such that $\fench(\sol,\temp \ybase) = \eps + \delta$.
Then, \eqref{eq:Fench-polar} gives
\begin{equation}
\label{eq:dy}
\dnorm{y_{0} - \ybase}
	= c \dnorm{\payv(\sol)}
	\geq \frac{K\norm{\feas}}{\temp} \left[ \sqrt{1 + 2\delta/(K\norm{\feas}^{2})} - 1 \right]
	\geq \frac{\delta}{2\temp\norm{\feas}},
\end{equation}
where we used the fact that $\delta < K\norm{\feas}^{2}$ in the last inequality.

To proceed, set $\tau_{\cone} = \inf\setdef{t\geq0}{Y(t) \notin \ybase + \cone}$ and let $G_{\genvec}(t) = \braket{Y(t) - \ybase}{\genvec}$, so $\tau_{\cone} = \inf\setdef{t\geq0}{\text{$G_{\genvec}(t) > 0$ for some $\genvec\in\genvecs$}}$.
Then, for all $t \leq \tau_{\cone}$, we have
\begin{flalign}
\label{eq:Gz-bound}
G_{\genvec}(t)
	&= G_{\genvec}(0) + \int_{0}^{t} \braket{\payv(X(s))}{\genvec} \dd s + \snoise_{\genvec}(t)
	\leq -A\norm{\genvec} - B\norm{\genvec} t + \snoise_{\genvec}(t),
\end{flalign}
where we have set
$A = c \min_{\genvec'\in\genvecs} \abs{\braket{\payv(\sol)}{\genvec'}}$,
$B = \sharp_{\nhd}$,
and $\snoise_{\genvec}(t) = \braket{Z(t)}{\genvec}$.
Arguing as in the proof of \cref{lem:Wbound}, the Dambis\textendash Dubins\textendash Schwarz time-change theorem for martingales \cite[Theorem~3.4.6]{KS98} implies that there exists a standard Wiener process $W_{\genvec}(t)$ such that $\snoise_{\genvec}(t) = W_{\genvec}(\rho_{\genvec}(t))$, where $\rho_{\genvec}(t) = [\snoise_{\genvec}(t),\snoise_{\genvec}(t)]$ denotes the quadratic variation of $\snoise_{\genvec}$.
By \eqref{eq:Gz-bound}, this further implies that $G_{\genvec}(t) \leq 0$ whenever $W_{\genvec}(\rho_{\genvec}(t)) \leq A\norm{\genvec} + B\norm{\genvec}t$;
hence, $\tau_{\cone} = \infty$ whenever $W_{\genvec}(\rho_{\genvec}(t)) \leq A\norm{\genvec} + B\norm{\genvec} t$.

Moreover, note that
\begin{equation}
d\rho_{\genvec}
	= d\snoise_{\genvec}\cdot d\snoise_{\genvec}
	= \sum_{i,j=1}^{n} \covmat_{ij} \genvec_{i} \genvec_{j} \dd t,
\end{equation}
so $\rho_{\genvec}(t) \leq \noisevar\norm{\genvec}^{2}_{2} t \leq \fradius\noisevar\norm{\genvec}^{2} t$ for some constant $\fradius>0$ that depends only on the choice of primal norm $\norm{\cdot}$.
Hence, if a trajectory of $W_{\genvec}$ is such that $W_{\genvec}(t) \leq A\norm{\genvec} + \frac{B}{\fradius\norm{\genvec}\noisevar}t$ for all $t\geq0$, we also get
\begin{equation}
W_{\genvec}(\rho_{\genvec}(t))
	\leq A\norm{\genvec} + \frac{B}{\fradius\norm{\genvec}\noisevar} \rho(t)
	\leq A\norm{\genvec} + B\norm{\genvec} t
	\quad
	\text{for all $t\geq0$}.
\end{equation}
Therefore, to prove the lemma, it suffices to establish a suitable lower bound for the probability $\probof{W_{\genvec}(t) \leq A\norm{\genvec} + Bt / (\fradius\norm{\genvec}\noisevar)\;\text{for all $t\geq0$}}$.

To do so, let
\begin{equation}
\tau_{\cone}'
	= \inf\setdef*{t>0}{\text{$W_{\genvec}(t) = A\norm{\genvec} + \frac{B}{\fradius\norm{\genvec}\noisevar}t$ for some $\genvec\in\genvecs$}}
\end{equation}
and write $E_{\genvec}$ for the event ``$W_{\genvec}(t) \geq A\norm{\genvec} + Bt/(\fradius\norm{\genvec}\noisevar)$ for some finite $t\geq0$''.
By a standard application of Girsanov's theorem for Wiener processes with drift \cite[p.~197]{KS98}, we get $\probof{E_{\genvec}} = e^{-2AB/(\fradius\noisevar)}$ and hence
\begin{equation}
\probof{\tau_{\cone}'<\infty}
	= \probof*{\union\nolimits_{\genvec\in\genvecs}E_{\genvec}}
	\leq \insum_{\genvec\in\genvecs} \probof{E_{\genvec}}
	= \abs{\genvecs} e^{-2AB/(\fradius\noisevar)}.
\end{equation}
Now, from the bound \eqref{eq:dy} and the definition of $A$ and $B$, we have
\begin{equation}
\frac{AB}{\fradius}
	= \frac{c\sharp_{\nhd} \min_{\genvec'\in\genvecs} \abs{\braket{\payv(\sol)}{\genvec'}}}{\fradius}
	\geq \frac{\delta}{2\temp\norm{\feas}}
		\frac{\sharp_{\nhd} \min_{\genvec'\in\genvecs} \abs{\braket{\payv(\sol)}{\genvec'}}}{\fradius\dnorm{\payv(\sol)}}
	= \frac{\kappa\delta}{\temp},
\end{equation}
where we set $\kappa = \frac{\sharp_{\nhd} \min_{\genvec'\in\genvecs} \abs{\braket{\payv(\sol)}{\genvec'}}}{2\fradius\dnorm{\payv(\sol)} \norm{\feas}}$.
Backtracking then yields $\probof{\tau_{\cone} = \infty} \geq \probof{\tau_{\cone}' = \infty} \geq 1 - e^{-\kappa\delta/(\temp\noisevar)}$, provided that $\temp\leq \kappa\delta/ (\noisevar \log\abs{\genvecs})$.
Therefore, with $\probof{Y(t) \in \ybase+\cone \; \text{for all} \; t\geq 0} = \probof{\tau_{\cone} = \infty}$, our proof is complete.
\end{proof}

The final ingredient of our proof is that if $Y(t)$ moves deep within $\pcone(\sol)$, the induced trajectory $X(t) = \mirror(\temp Y(t))$ converges to $\sol$:

\begin{lemma}
\label{lem:conv-polar}
Let $(y_{n})_{n=1}^{\infty}$ be a sequence in $\dual$ such that $\braket{y_{n}}{z} \to -\infty$ for all $z\in\tcone(\sol)$.
Then, $\lim \mirror(y_{n}) = \sol$.
\end{lemma}

\begin{proof}
By compactness of $\feas$ (and passing to a subsequence if necessary), we may assume that $x_{n} \equiv \mirror(y_{n})$ converges in $\feas$.
Assume therefore that $x_{n} \to x' \neq \sol$, so $\liminf \norm{x_{n} - \sol} > 0$.
Then, with $y_{n} \in \subd h(x_{n})$ by \cref{prop:mirror}, we get
\begin{flalign}
h(\sol)
	\geq h(x_{n}) + \braket{y_{n}}{\sol - x_{n}}
	\geq h(x_{n}) - \braket{y_{n}}{z_{n}} \norm{x_{n} - \sol},
\end{flalign}
where we set $z_{n} = (x_{n} - \sol) / \norm{x_{n} - \sol}$.
Since $z_{n}$ lives in the unit sphere of $\norm{\argdot}$, compactness (and a descent to a further subsequence if necessary) guarantees the existence of some $z\in\tcone(\sol)$ with $\norm{z} = 1$ and such that $\braket{y_{n}}{z_{n}} \leq \braket{y_{n}}{z}$ for all $n$ (recall that $\tcone(\sol)$ is closed).
We thus get $h(\sol) \geq h(x_{n}) - \braket{y_{n}}{z} \norm{x_{n} - \sol}$ and, taking $\liminf$ on both sides, we obtain $\liminf h(\sol) = \infty$, a contradiction.
\end{proof}

We are now in a position to prove our main result for sharp solutions:

\begin{proofof}{Proof of \cref{thm:sharp}}
As in the proof of \cref{lem:escape}, let $\nhd$ be a sufficiently small compact neighborhood of $\sol$ such that $\payv(\nhd)\subseteq\intr(\pcone(\sol))$, i.e. $\braket{\payv(x)}{z} \leq - \sharp_{\nhd} \norm{z}$ for some $\sharp_{\nhd}>0$ and for all $x\in \nhd$, $z\in\tcone(\sol)$.
Then, by compactness, there exists a convex cone $\cone\subseteq\intr(\pcone(\sol))$ such that $\braket{\payv(x)}{z} \leq -\sharp_{\nhd} \norm{z}$ for all $x\in \nhd$, $z\in\cone^{\perp}$.

With this in mind, pick $\eps,\delta>0$ sufficiently small so that the conclusion of \cref{lem:escape} holds and $\mirror(\temp y) \in \nhd$ whenever $\fench(\sol,\temp y) \leq \eps + \delta$.
If $\temp$ is also chosen small enough, combining \eqref{eq:Fench-reg} with \cref{lem:sharp-recurrence} shows that there exists a (random) sequence of times $t_{n}\uparrow\infty$ such that $\fench(\sol,\temp Y(t_{n})) \leq \eps$ for all $n$ \as.
Hence, by \cref{lem:escape} and the strong Markov property of $Y(t)$, there exists some $a>0$ such that $\probof{\text{$\fench(\sol,\temp Y(t_{n}+t)) \leq \eps + \delta$ for all $t\geq 0$}} \geq 1 - (1-a)^{n}$ for all $n$.
Thus, with notation as in \eqref{eq:Gz-bound}, we get
\begin{equation}
G_{z}(t_{n} + t)
	\leq -A\norm{z} - B\norm{z} t + \snoise_{z}(t)
	\quad
	\text{for all $t\geq0$},
\end{equation}
with probability at least $1 - (1-a)^{n}$.
In turn, \cref{lem:Wbound} yields $\snoise_{z}(t)/t\to0$ \as, showing that $\lim_{t\to\infty} G_{z}(t_{n} + t) = -\infty$.
Since the above holds for all $n$, we conclude that $\braket{Y(t)}{z} \to -\infty$ for all $z\in\tcone(\sol)$, so $X(t)\to\sol$ \as by \cref{lem:conv-polar}.

We are left to show that this convergence occurs in finite time if $\mirror$ is surjective.
To that end, note first that if $\sol = \mirror(\temp\dsol)$, we also have $\mirror(\temp(\dsol + \payv)) = \sol$ for all $\payv\in\pcone(\sol)$ by \cref{lem:inv-cone}.
Therefore, it suffices to show that, for some $\dsol$ such that $\mirror(\temp\dsol) = \sol$, we have $Y(t) \in\dsol + \pcone(\sol)$ for all sufficiently large $t$ \as.
However, since $X(t)\to\sol$, there exists some $t_{0}$ such that $X(t)\in\nhd$ for all $t\geq t_{0}$.
Thus, for all $z\in\tcone(\sol)$ with $\norm{z}=1$, we get
\begin{equation}
\braket{Y(t) - Y(t_{0})}{z}
	= \int_{t_{0}}^{t} \braket{\payv(X(s))}{z} \dd s
	+ \braket{\noise(t)}{z}
	\leq - \sharp_{\nhd} (t - t_{0}) + \dnorm{\noise(t)}.
\end{equation}
Since $\noise(t)/t\to 0$ by \cref{lem:Wbound}, we conclude that $\braket{Y(t)}{z} \to -\infty$ uniformly in $z$ \as.
Consequently, there exists some $t_{0}'$ such that $\braket{Y(t) - \dsol}{z} \leq 0$ for all $t\geq t_{0}'$ and all $z\in\tcone(\sol)$ with $\norm{z} = 1$.
In turn, this implies that $Y(t) \in \dsol + \pcone(\sol)$ for all $t\geq t_{0}'$ and our proof is complete.
\end{proofof}

\subsection{Convergence via rectification}
\label{app:transforms}

We now turn to the rectified variants of \eqref{eq:SMD} with a decreasing sensitivity parameter:

\begin{proofof}{Proof of \cref{thm:conv-var}}
For all $x\in\feas$ and $\sol\in\feas$, we have
\begin{equation}
\obj(x) - \min\obj
	= \obj(x) - \obj(\sol)
	\leq \braket{\nabla\obj(x)}{x - \sol}
	= \braket{\payv(x)}{\sol - x},
\end{equation}
by convexity of $\obj$.
Hence, by the definition of $\wilde X$ (and Jensen's inequality in the case of \eqref{eq:X-avg}), we obtain
\begin{equation}
\label{eq:obj-regret}
\obj(\wilde X(t))
	\leq \min\obj + \frac{1}{t}\int_{0}^{t} \braket{\payv(X(s))}{\sol - X(s)} \dd s,
\end{equation}
so it suffices to properly majorize the \acl{RHS} of the above equation.

To that end, let $\lyap(t) = \temp(t)^{-1} \fench(\sol,\temp(t) Y(t))$ denote the $\temp$-deflated Fenchel coupling between $Y(t)$ and $\sol\in\argmin\obj$.
Then, \cref{lem:dFench-stoch} yields
\begin{subequations}
\label{eq:rate-var1}
\begin{flalign}
\int_{0}^{t} \braket{\payv(X(s))}{\sol - X(s)} \dd s
	&\label{eq:rate-Fench}
	\leq \lyap(0) - \lyap(t)
	\\
	&\label{eq:rate-temp}
	-\int_{0}^{t} \frac{\dot\temp(s)}{\temp^{2}(s)} \left[ h(\sol) - h(X(s)) \right] ds
	\\
	&\label{eq:rate-Ito}
	+ \frac{1}{2K} \int_{0}^{t} \temp(s) \trof{\covmat(X(s),s)} \dd s
	\\
	&\label{eq:rate-noise}
	+ \sum_{i=1}^{n} \int_{0}^{t} (X_{i}(s) - \sol_{i}) \dd Z_{i}(s).
\end{flalign}
\end{subequations}

We now proceed to bound each term of \eqref{eq:rate-var1}:
\begin{enumerate}
[\itshape a\upshape)]
\addtolength{\itemsep}{.5ex}

\item
Since $\lyap(t)\geq0$ for all $t$, the term \eqref{eq:rate-Fench} is bounded from above by $\lyap(0)$, viz.
\begin{equation}
\label{eq:rate-Fench1}
\eqref{eq:rate-Fench}
	\leq \lyap(0)
	= \frac{h(\sol) + h^{\ast}(\temp(0) Y(0))}{\temp(0)} - \braket{Y(0)}{\sol}
\end{equation}

\item
For \eqref{eq:rate-temp},
we have $h(\sol) - h(X(s)) \leq \depth$ by definition, so, with $\dot\temp(t) \leq 0$ for almost all $t$ by \eqref{eq:temp}, we get
\begin{equation}
\label{eq:rate-temp1}
\eqref{eq:rate-temp}
	\leq - \depth \int_{0}^{t} \frac{\dot\temp(s)}{\temp^{2}(s)} \dd s
	= \frac{\depth}{\temp(t)} - \frac{\depth}{\temp(0)}.
\end{equation}

\item
For \eqref{eq:rate-Ito}, the definition of $\noisevar$ gives
\(
\eqref{eq:rate-Ito}
	\leq (2K)^{-1} \noisevar \int_{0}^{t} \temp(s) \dd s.
\)

\item
Finally, for \eqref{eq:rate-noise}, let $\snoise(t) = \int_{0}^{t} \sum_{i=1}^{n} (X_{i}(s) - \sol_{i}) \dd Z_{i}(s)$ and write $\rho(t) = [\snoise(t),\snoise(t)]$ for the quadratic variation of $\snoise$.
We then get
\begin{flalign}
d[\snoise,\snoise]
	= d\snoise \cdot d\snoise
	&= \sum_{i,j=1}^{n} \covmat_{ij} (X_{i} - \sol_{i}) (X_{j} - \sol_{j}) \dd t
	\leq \noisevar \norm{X - \sol}^{2}_{2} \dd t,
\end{flalign}
so $\rho(t) \leq \fradius\noisevar \norm{\feas}^{2} t$ for some norm-dependent constant $\fradius>0$.
Arguing as in the proof of \cref{lem:Wbound} in \cref{app:stoch}, the Dam\-bis\textendash Du\-bins\textendash Schwarz time-change theorem for martingales \cite[Theorem~3.4.6 and Problem~3.4.7]{KS98} shows that there exists a one-dimensional Wiener process $\wilde W(t)$ with induced filtration $\wilde\filter_{s} = \filter_{\tau_{\rho}(s)}$ and such that $\wilde W(\rho(t)) = \snoise(t)$ for all $t\geq0$.
By the law of the iterated logarithm \cite[p.~112]{KS98}, we then obtain
\begin{equation}
\label{eq:loglog}
\limsup_{t\to\infty} \frac{\wilde W(\rho(t))}{\sqrt{2 Mt \log \log (Mt)}}
	 \leq \limsup_{t\to\infty} \frac{\wilde W(\rho(t))}{\sqrt{2 \rho(t) \log \log \rho(t)}}
	= 1
	\quad
	\text{(a.s.),}
\end{equation}
where $M = \fradius\noisevar \norm{\feas}^{2}$.
Thus, with probability $1$, we have $\snoise(t) = \bigoh(\sqrt{t \log\log t})$.

\end{enumerate}

Putting together all of the above and dividing by $t$, we get
\begin{equation}
\label{eq:rate-var2}
\frac{1}{t} \int_{0}^{t} \braket{\payv(X(s))}{\sol - X(s)} \dd s
	\leq \frac{\depth}{t\temp(t)}
	+ \frac{\noisevar}{2Kt} \int_{0}^{t} \temp(s) \dd s
	+ \bigoh(t^{-1/2} \sqrt{\log\log t}),
\end{equation}
where we have absorbed the $\bigoh(1/t)$ terms from \eqref{eq:rate-Fench1} and \eqref{eq:rate-temp1} in the logarithmic term $\bigoh(\sqrt{t^{-1}\log\log t})$.
Our claim then follows from \eqref{eq:obj-regret}.
Finally,
recalling that $\snoise(t)$ is a zero-mean local martingale,
the mean bound \eqref{eq:rate-mean} follows by taking expectations above.
\end{proofof}

\section{Results from stochastic analysis}
\label{app:stoch}

In this last appendix, we collect some results from stochastic analysis that we use throughout the paper.
The first such result is a growth estimate for Itô martingales with bounded volatility:

\begin{lemma}
\label{lem:Wbound}
Let $W(t)$ be a Wiener process in $\R^{m}$ and let $\zeta(t)$ be a bounded, continuous process in $\R^{m}$.
Then, for every function $f\from[0,\infty)\to(0,\infty)$, we have
\begin{equation}
\label{eq:Wbound}
f(t) + \int_{0}^{t} \zeta(s) \cdot dW(s)
	\sim f(t)
	\quad
	\text{as $t\to\infty$ \textup(a.s.\textup),}
\end{equation}
whenever
$\lim_{t\to\infty} \left(t\log\log t\right)^{-1/2} f(t) = +\infty$.
\end{lemma}

\begin{proofof}{Proof of \cref{lem:Wbound}}
Let $\snoise(t) = \sum_{i=1}^{n} \int_{0}^{t} \zeta_{i}(s) \dd W_{i}(s)$.
Letting $\rho(t) = [\snoise(t),\snoise(t)]$ denote the quadratic variation of $\snoise(t)$, we have
\begin{equation}
\label{eq:covest1}
d\rho
	= \insum_{i=1}^{n} \zeta_{i} \zeta_{j} \delta_{ij} \dd t
	\leq M \dd t,
\end{equation}
where $M = \sup_{t\geq0} \norm{\zeta(t)}^{2} < \infty$ (recall that $\zeta(t)$ is bounded by assumption).
Now, let $\rho_{\infty} = \lim_{t\to\infty} \rho(t) \in[0,\infty]$ and set
\begin{equation}
\tau_{\rho}(s) = \begin{cases}
	\inf\setdef{t\geq0}{\rho(t) > s}
		&\quad
		\text{if $s\leq\rho_{\infty}$},
		\\
	\infty
		&\quad
		\text{otherwise}.
	\end{cases}
\end{equation}
The process $\tau_{\rho}(s)$ is finite, non-negative, non-decreasing, and right-continuous on $[0,\rho_{\infty})$;
moreover, it is easy to check that $\rho(\tau_{\rho}(s)) = s \wedge \rho_{\infty}$ and $\tau_{\rho}(\rho(t)) = t$ \cite[Problem~3.4.5]{KS98}.
Therefore, by the Dambis\textendash Dubins\textendash Schwarz time-change theorem for martingales \cite[Thm.~3.4.6 and Pb.~3.4.7]{KS98}, there exists a standard, one-dimensional Wiener process $\wilde W(t)$ with induced filtration $\wilde\filter_{s} = \filter_{\tau_{\rho}(s)}$ and such that $\wilde W(\rho(t)) = \snoise(t)$ for all $t\geq0$.
The rest of the proof then follows by applying the law of the iterated logarithm as in \cite[Lemma~B.4]{BM17}. 
%
\end{proofof}

The second result we report here is a weak version of Itô's formula for differentiable functions with Lipschitz-continuous gradient.
For notational convenience, let $\mathbf{C}_{\Lip}^{1,1}(\dual)$ denote the space of functions $\test\from\dual\to\R$ such that
\begin{equation}
\norm{\nabla \test(y_{2}) - \nabla\test(y_{1})}
	\leq \Lip\norm{y_{2} - y_{1}}_{\ast}
	\quad
	\text{for all $y_{1},y_{2}\in\dual$}.
\end{equation}
We then have:

\begin{proposition}
\label{prop:Ito}
Let $Y(t) = (Y_{i}(t))_{i=1}^{n}$ be a $\dual$-valued Itô process of the form 
\begin{equation}
Y_{i}(t)
	= Y_{i}(0) + \int_{0}^{t} \alpha_{i}(s) \dd s
	+ \sum_{k=1}^{m} \int_{0}^{t} \beta_{ik}(s) \dd W_{k}(s),
\end{equation}
where $W(t) = (W_{k}(t))_{k=1}^{m}$ is a standard $m$-dimensional Wiener process.
If $\test\in\mathbf{C}^{1,1}_{\Lip}(\dual)$ is convex, then, for all $t\geq0$, we have:
\begin{equation}
\test(Y(t))
	\leq\test(Y_{0})
	+\int_{0}^{t} \braket{\nabla\test(Y(s))}{dY(s)}
	+\frac{\Lip}{2} \int_{0}^{t} \trof{\beta(s)\beta(s)^{\top}} \dd s.
\end{equation}
\end{proposition}

The proof of \cref{prop:Ito} is based on the following property of convex functions in $\mathbf{C}^{1,1}_{\Lip}(\R^{n})$:

\begin{lemma}
\label{lem:Hessianbound}
Let $\test\in\mathbf{C}^{1,1}_{\Lip}(\dual)$ be convex.
Then $\test$ is almost everywhere twice differentiable and its Hessian satisfies
\begin{equation}
0
	\mleq \hess(\test(y))
	\mleq \Lip I
	\quad
	\text{for \textpar{Lebesgue} almost all $y\in\dual$.}
\end{equation}
\end{lemma}

\begin{proof}
The fact that $\test$ is twice differentiable (Lebesgue) a.e. is Alexandrov's theorem.
Hence, there exists a Lebesgue-full set $\dual_{0}\subseteq\dual$ such that
\begin{equation}
\label{eq:psi1}
\test(y + z)
	= \test(y)
	+ \braket{\nabla\test(y)}{z}
	+ \frac{1}{2} z^{\top} \hess(\test(y)) z
	+ \theta(y,z)
	\quad
	\text{for all $y\in\dual_{0}$},
\end{equation}
with $\theta(y,z) = o(\dnorm{z^{2}})$.
Furthermore, by the well-known descent lemma for functions with Lipschitz continuous gradient \cite[Theorem~2.1.5]{Nes04}, we also have
\begin{equation}
\label{eq:descent}
\test(y+z)
	\leq \test(y)
	+ \braket{\nabla\test(y)}{z}
	+ \frac{\Lip}{2} \norm{z}_{\ast}^{2}
	\quad
	\text{for all $y,z \in\dual$}. 
\end{equation}
Thus, taking $z=t\dvec$ for some unit vector $\dvec\in\dual$ (i.e. $\dnorm{\dvec}=1$) and combining the above, we readily obtain
\begin{equation}
\frac{t^{2}}{2} \dvec^{\top} \hess(\test(y))\dvec
	+ \theta(y,t\dvec)
	\leq \frac{\Lip}{2} t^{2}
	\quad
	\text{for all $y\in\dual_{0}$, $t\geq0$}.
\end{equation}
Hence, dividing by $t$ and letting $t\to0^{+}$ yields
\begin{equation}
\dvec^{\top} \hess(\test(y)) \dvec
	\leq \frac{\Lip}{2}
	\quad
	\text{for all $y\in\dual_{0}$},
\end{equation}
implying in turn that $\hess(\test(y)) \mleq \Lip I$ for all $y\in\dual_{0}$.
The bound $\hess(\test(y)) \mgeq 0$ is a trivial consequence of convexity, completing our proof.
\end{proof}

\begin{proofof}{Proof of \cref{prop:Ito}}
Our proof relies on smoothing by mollification.
To begin, consider the standard unit mollifier 
\begin{equation}
\rho(u)
	= \begin{cases}
	c\exp\left(-\frac{1}{1-\norm{u}_{\ast}^{2}}\right)
		&\text{ if }\norm{u}_{\ast}<1,
		\\
	0
		& \text{if }\norm{u}_{\ast}\geq 1,
	\end{cases}
\end{equation}
with $c>0$ chosen so that $\int_{\R^{n}} \rho(w) \dd w=1$.
Then, for all $\eps>0$, let 
\begin{subequations}
\begin{flalign}
\rho_{\eps}(u)
	&=\eps^{-n}\rho(u/\eps),
\intertext{and}
\test_{\eps}(y)
	&= (\test\ast\rho_{\eps})(y)
	=\int_{\dual} \test(y-w) \rho_{\eps}(w) \dd w,
\end{flalign}
\end{subequations}
with ``$\ast$'' above denoting convolution over $\R^{n}$.
We then have $\test_{\eps}\in\mathbf{C}^{\infty}(\dual)$, so the standard form of Itô's formula gives us
\begin{flalign}
\test_{\eps}(Y(t))
	&= \test_{\eps}(Y(s))
	+ \int_{s}^{t}\braket{\nabla\test_{\eps}(Y(\tau))}{dY(\tau)}
	\notag\\
	&+ \frac{1}{2} \int_{s}^{t} \trof*{\hess(\test_{\eps}(Y(\tau))) \beta(\tau)\beta(\tau)^{\top}} \dd\tau
	\notag\\
	&= \test_{\eps}(Y(s))
	+ \int_{s}^{t} \braket*{\int_{\dual} \nabla\test(z)\rho_{\eps}(Y(\tau) - z) \dd z}{dY(\tau)}
	\notag\\
	&+\frac{1}{2} \int_{s}^{t}\int_{\dual} \trof*{\hess(\test(z)) \beta(\tau)\beta(\tau)^{\top}} \rho_{\eps}(Y(\tau) - z) \dd\tau \dd z,
\end{flalign}
where the last equality uses the fact that $\hess(\test)$ exists for (Lebesgue) almost all $y$, as established in \cref{lem:Hessianbound}.
Using \cref{lem:Hessianbound} one more time, we further have $\trof{\hess(\test(z))\beta(\tau)\beta(\tau)^{\top}} \leq \Lip\trof{\beta(\tau)\beta(\tau)^{\top}}$, implying in turn that
\begin{flalign}
\test_{\eps}(Y(t)) - \test_{\eps}(Y(s))
	&\leq \int_{s}^{t} \braket*{\int_{\dual} \nabla\test(z) \rho_{\eps}(Y(\tau) - z) \dd z}{dY(\tau)}
	\notag\\
	&+ \frac{\Lip}{2} \int_{s}^{t} \trof{\beta(\tau)\beta(\tau)^{\top}} \dd\tau. 
\end{flalign}
Our assertion then follows by letting $\eps\to0^{+}$ and invoking the dominated convergence theorem.
\end{proofof}

\begin{remark}
\label{rem:Ito}
In the main body of the paper, the above result is typically applied to the Fenchel coupling $\fench(\base,y)$ which, as a function of $y$, is in the class $\mathbf{C}^{1,1}_{1/K}(\dual)$ for every $\base\in\feas$, by \cref{prop:mirror}.
Specifically, letting $Y(t)$ denote the unique strong solution to \eqref{eq:SMD} and taking $\lyap(t) = \fench(\base,Y(t))$ for some $\base\in\feas$, \cref{prop:Ito} yields 
\begin{flalign}
\lyap(t) - \lyap(0)
	&\leq \int_{0}^{t}\braket{\nabla \fench(\base,Y(s))}{dY(s)}
	+\frac{1}{2K} \int_{0}^{t}\trof{\covmat(X(s),s)} \dd s,
\end{flalign}
where we used the definition $\covmat = \noisevol\noisevol^{\top}$ of $\covmat$.
\end{remark}

\bibliographystyle{siam}
\bibliography{IEEEabrv,Bibliography-SMD}

\begin{thebibliography}{10}

\bibitem{AA15}
{\sc B.~Abbas and H.~Attouch}, {\em Dynamical systems and forward-backward
  algorithms associated with the sum of a convex subdifferential and a monotone
  cocoercive operator}, Optimization, 64 (2015), pp.~2223--2252.

\bibitem{AABR02}
{\sc F.~Alvarez, H.~Attouch, J.~Bolte, and P.~Redont}, {\em A second-order
  gradient-like dissipative dynamical system with {Hessian} damping.
  {Applications} to optimization and mechanics}, Journal des Math\'ematiques
  Pures et Appliqu\'ees, 81 (2002), pp.~774--779.

\bibitem{ABB04}
{\sc F.~Alvarez, J.~Bolte, and O.~Brahic}, {\em Hessian {Riemannian} gradient
  flows in convex programming}, SIAM Journal on Control and Optimization, 43
  (2004), pp.~477--501.

\bibitem{ABRT04}
{\sc H.~Attouch, J.~Bolte, P.~Redont, and M.~Teboulle}, {\em Singular
  {Riemannian} barrier methods and gradient-projection dynamical systems for
  constrained optimization}, Optimization, 53 (2004), pp.~435--454.

\bibitem{AGR00}
{\sc H.~Attouch, X.~Goudou, and P.~Redont}, {\em The heavy ball with friction
  method, {I}. {T}he continuous dynamical system: global exploration of the
  local minima of a real-valued function by asymptotic analysis of a
  dissipative dynamical system}, Communications in Contemporary Mathematics, 2
  (2000), pp.~1--34.

\bibitem{BMW56}
{\sc M.~Beckmann, C.~B. McGuire, and C.~Winsten}, {\em Studies in the Economics
  of Transportation}, Yale University Press, 1956.

\bibitem{BBJ15}
{\sc P.~B{\'e}gout, J.~Bolte, and M.-A. Jendoubi}, {\em On damped second-order
  gradient systems}, Journal of Differential Equations, 259 (2015),
  pp.~3115--3143.

\bibitem{Ben99}
{\sc M.~Bena{\"\i}m}, {\em Dynamics of stochastic approximation algorithms}, in
  S{\'e}minaire de Probabilit{\'e}s XXXIII, J.~Az{\'e}ma, M.~{\'E}mery,
  M.~Ledoux, and M.~Yor, eds., vol.~1709 of Lecture Notes in Mathematics,
  Springer Berlin Heidelberg, 1999, pp.~1--68.

\bibitem{BH96}
{\sc M.~Bena{\"\i}m and M.~W. Hirsch}, {\em Asymptotic pseudotrajectories and
  chain recurrent flows, with applications}, Journal of Dynamics and
  Differential Equations, 8 (1996), pp.~141--176.

\bibitem{BG92}
{\sc D.~P. Bertsekas and R.~Gallager}, {\em Data Networks}, Prentice Hall,
  Englewood Cliffs, NJ, 2~ed., 1992.

\bibitem{Bha78}
{\sc R.~N. Bhattacharya}, {\em Criteria for recurrence and existence of
  invariant measures for multidimensional diffusions}, Annals of Probability,
  (1978), pp.~541--553.

\bibitem{BH16}
{\sc P.~Bianchi and W.~Hachem}, {\em Dynamical behavior of a stochastic
  forward--backward algorithm using random monotone operators}, Journal of
  Optimization Theory and Applications, 171 (2016), pp.~90--120.

\bibitem{BC16}
{\sc R.~I. Bo\c{t} and E.~R. Csetnek}, {\em Approaching the solving of
  constrained variational inequalities via penalty term-based dynamical
  systems}, Journal of Mathematical Analysis and Applications, 435 (2016),
  pp.~1688--1700.

\bibitem{BT03}
{\sc J.~Bolte and M.~Teboulle}, {\em Barrier operators and associated
  gradient-like dynamical systems for constrained minimization problems}, SIAM
  Journal on Control and Optimization, 42 (2003), pp.~1266--1292.

\bibitem{BM17}
{\sc M.~Bravo and P.~Mertikopoulos}, {\em On the robustness of learning in
  games with stochastically perturbed payoff observations}, Games and Economic
  Behavior, to appear (2017).

\bibitem{CEG09}
{\sc A.~Cabot, H.~Engler, and S.~Gadat}, {\em On the long time behavior of
  second order differential equations with asymptotically small dissipation},
  Transactions of the American Mathematical Society, 361 (2009),
  pp.~5983--6017.

\bibitem{CT93}
{\sc G.~Chen and M.~Teboulle}, {\em Convergence analysis of a proximal-like
  minimization algorithm using {Bregman} functions}, SIAM Journal on
  Optimization, 3 (1993), pp.~538--543.

\bibitem{DAJJ12}
{\sc J.~C. Duchi, A.~Agarwal, M.~Johansson, and M.~I. Jordan}, {\em Ergodic
  mirror descent}, SIAM Journal on Optimization, 22 (2012), pp.~1549--1578.

\bibitem{GP14}
{\sc S.~Gadat and F.~Panloup}, {\em Long time behaviour and stationary regime
  of memory gradient diffusions}, Annales de l'Institut Henri Poincar{\'e},
  Probabilit{\'e}s et Statistiques, 50 (2014), pp.~564--601.

\bibitem{HM96}
{\sc U.~Helmke and J.~B. Moore}, {\em Optimization and Dynamical Systems},
  Springer-Verlag, 1996.

\bibitem{HUL01}
{\sc J.-B. Hiriart-Urruty and C.~Lemar{\'e}chal}, {\em Fundamentals of Convex
  Analysis}, Springer, Berlin, 2001.

\bibitem{HS98}
{\sc J.~Hofbauer and K.~Sigmund}, {\em Evolutionary Games and Population
  Dynamics}, Cambridge University Press, Cambridge, UK, 1998.

\bibitem{Imh05}
{\sc L.~A. Imhof}, {\em The long-run behavior of the stochastic replicator
  dynamics}, The Annals of Applied Probability, 15 (2005), pp.~1019--1045.

\bibitem{ISdCN99}
{\sc A.~N. Iusem, B.~F. Svaiter, and J.~X. da~Cruz~Neto}, {\em Central paths,
  generalized proximal point methods, and {Cauchy} trajectories in {Riemannian}
  manifolds}, SIAM Journal on Control and Optimization, 37 (1999),
  pp.~566--588.

\bibitem{KSST12}
{\sc S.~M. Kakade, S.~Shalev-Shwartz, and A.~Tewari}, {\em Regularization
  techniques for learning with matrices}, The Journal of Machine Learning
  Research, 13 (2012), pp.~1865--1890.

\bibitem{KS98}
{\sc I.~Karatzas and S.~E. Shreve}, {\em Brownian Motion and Stochastic
  Calculus}, Springer-Verlag, Berlin, 1998.

\bibitem{Kar90}
{\sc N.~Karmarkar}, {\em Riemannian geometry underlying interior point methods
  for linear programming}, in Mathematical Developments Arising from Linear
  Programming, no.~114 in Contemporary Mathematics, American Mathematical
  Society, 1990.

\bibitem{Kha12}
{\sc R.~Z. Khasminskii}, {\em Stochastic Stability of Differential Equations},
  no.~66 in Stochastic Modelling and Applied Probability, Springer-Verlag,
  Berlin, 2~ed., 2012.

\bibitem{Kiw97b}
{\sc K.~C. Kiwiel}, {\em Free-steering relaxation methods for problems with
  strictly convex costs and linear constraints}, Mathematics of Operations
  Research, 22 (1997), pp.~326--349.

\bibitem{KE92}
{\sc P.~E. Kloeden and E.~Platen}, {\em Numerical Solution of Stochastic
  Differential Equations}, Springer Berlin Heidelberg, 1992.

\bibitem{Kri16}
{\sc W.~Krichene}, {\em Continuous and discrete dynamics for online learning
  and convex optimization}, PhD thesis, Department of Electrical Engineering
  and Computer Sciences, University of California, Berkeley, 2016.

\bibitem{KBB15}
{\sc W.~Krichene, A.~Bayen, and P.~Bartlett}, {\em Accelerated mirror descent
  in continuous and discrete time}, in NIPS '15: Proceedings of the 29th
  International Conference on Neural Information Processing Systems, 2015.

\bibitem{KM17}
{\sc J.~Kwon and P.~Mertikopoulos}, {\em A continuous-time approach to online
  optimization}, Journal of Dynamics and Games, 4 (2017), pp.~125--148.

\bibitem{LTE15}
{\sc Q.~Li, C.~Tai, and W.~E}, {\em Dynamics of stochastic gradient
  algorithms}.
\newblock \url{https://arxiv.org/abs/1511.06251}, 2015.

\bibitem{Mer17}
{\sc P.~Mertikopoulos}, {\em Learning in games with continuous action sets and
  unknown payoff functions}.
\newblock \url{https://arxiv.org/abs/1608.07310}, 2016.

\bibitem{MBNS17}
{\sc P.~Mertikopoulos, E.~V. Belmega, R.~Negrel, and L.~Sanguinetti}, {\em
  Distributed stochastic optimization via matrix exponential learning}, {IEEE}
  Trans. Signal Process., 65 (2017), pp.~2277--2290.

\bibitem{MM10}
{\sc P.~Mertikopoulos and A.~L. Moustakas}, {\em The emergence of rational
  behavior in the presence of stochastic perturbations}, The Annals of Applied
  Probability, 20 (2010), pp.~1359--1388.

\bibitem{MS16}
{\sc P.~Mertikopoulos and W.~H. Sandholm}, {\em Learning in games via
  reinforcement and regularization}, Mathematics of Operations Research, 41
  (2016), pp.~1297--1324.

\bibitem{MV16}
{\sc P.~Mertikopoulos and Y.~Viossat}, {\em Imitation dynamics with payoff
  shocks}, International Journal of Game Theory, 45 (2016), pp.~291--320.

\bibitem{NJLS09}
{\sc A.~S. Nemirovski, A.~Juditsky, G.~G. Lan, and A.~Shapiro}, {\em Robust
  stochastic approximation approach to stochastic programming}, SIAM Journal on
  Optimization, 19 (2009), pp.~1574--1609.

\bibitem{NY83}
{\sc A.~S. Nemirovski and D.~B. Yudin}, {\em Problem Complexity and Method
  Efficiency in Optimization}, Wiley, New York, NY, 1983.

\bibitem{Nes04}
{\sc Y.~Nesterov}, {\em Introductory Lectures on Convex Optimization: A Basic
  Course}, no.~87 in Applied Optimization, Kluwer Academic Publishers, 2004.

\bibitem{Nes09}
\leavevmode\vrule height 2pt depth -1.6pt width 23pt, {\em Primal-dual
  subgradient methods for convex problems}, Mathematical Programming, 120
  (2009), pp.~221--259.

\bibitem{ParRas14}
{\sc E.~Pardoux and A.~Rascanu}, {\em Stochastic Differential Equations,
  Backwards SDE, Partial Differential Equations}, Springer - Stochastic
  Modelling and Applied Probability, 2014.

\bibitem{Pol87}
{\sc B.~T. Polyak}, {\em Introduction to Optimization}, Optimization Software,
  1987.

\bibitem{RB13}
{\sc M.~Raginsky and J.~Bouvrie}, {\em Continuous-time stochastic mirror
  descent on a network: Variance reduction, consensus, convergence}, in CDC
  '13: Proceedings of the 51st IEEE Annual Conference on Decision and Control,
  2013.

\bibitem{Rob95}
{\sc C.~Robinson}, {\em Dynamical Systems: Stability, Symbolic Dynamics, and
  Chaos}, CRC Press, Boca Raton, FL, 1995.

\bibitem{Roc70}
{\sc R.~T. Rockafellar}, {\em Convex Analysis}, Princeton University Press,
  Princeton, NJ, 1970.

\bibitem{RW98}
{\sc R.~T. Rockafellar and R.~J.~B. Wets}, {\em Variational Analysis}, vol.~317
  of A Series of Comprehensive Studies in Mathematics, Springer-Verlag, Berlin,
  1998.

\bibitem{SS11}
{\sc S.~Shalev-Shwartz}, {\em Online learning and online convex optimization},
  Foundations and Trends in Machine Learning, 4 (2011), pp.~107--194.

\bibitem{SNW12}
{\sc S.~Sra, S.~Nowozin, and S.~J. Wright}, {\em Optimization for Machine
  Learning}, MIT Press, Cambridge, MA, USA, 2012.

\bibitem{SBC14}
{\sc W.~Su, S.~Boyd, and E.~J. Cand{\`e}s}, {\em A differential equation for
  modeling {Nesterov}'s accelerated gradient method: Theory and insights}, in
  NIPS '14: Proceedings of the 27th International Conference on Neural
  Information Processing Systems, 2014, pp.~2510--2518.

\bibitem{Ved02}
{\sc V.~Vedral}, {\em The role of relative entropy in quantum information
  theory}, Reviews of Modern Physics, 74 (2002), pp.~197--234.

\bibitem{WWJ16}
{\sc A.~Wibisono, A.~C. Wilson, and M.~I. Jordan}, {\em A variational
  perspective on accelerated methods in optimization}.
\newblock \url{https://arxiv.org/abs/1603.04245}, 2016.

\end{thebibliography}

\end{document}